\documentclass[11pt,reqno]{article}
\usepackage{amsmath}
\usepackage{amsthm}
\usepackage{amssymb}
\usepackage{amsfonts}
\usepackage{cases}
\usepackage{graphicx}
\usepackage{xcolor}
\usepackage[margin=1in]{geometry}
\usepackage[utf8]{inputenc}
\usepackage[margin=1in]{geometry}
\usepackage{verbatim}
\usepackage{caption,subcaption}
\usepackage[colorlinks,citecolor=blue,urlcolor=blue]{hyperref}
\usepackage{bm}

\usepackage{algorithm}
\usepackage{algpseudocode}
\usepackage[normalem]{ulem}
\makeatletter
\renewcommand{\fnum@algorithm}{\fname@algorithm}
\allowdisplaybreaks

\numberwithin{equation}{section}
\newtheorem{Definition}{Definition}[section]
\newtheorem{Remark}{Remark}[section]
\newtheorem{Theorem}{Theorem}[section]
\newtheorem{Lemma}{Lemma}[section]
\newtheorem{Proposition}{Proposition}[section]

\newtheorem{Assumption}{Assumption}[section]

\newcommand{\be}{\begin{equation}}
\newcommand{\ee}{\end{equation}}
\newcommand{\bee}{\begin{equation*}}
\newcommand{\eee}{\end{equation*}}
\newcommand{\bi}{\begin{itemize}}
\newcommand{\ei}{\end{itemize}}
\DeclareMathOperator{\tr}{tr}
\DeclareMathOperator*{\argmax}{arg\,max}

\usepackage{algorithm}
\usepackage{algpseudocode}
\usepackage{soul}%delete line

\usepackage{enumitem}

\usepackage{accents}

\def \rightarrow{{\to}}
\def \E{\mathbb{E}}

\def \N{\mathbb{N}}
\def \P{\mathbb{P}}

\def \R{\mathbb{R}}

\def \T{\mathbb{T}}

\def \Cc{{\mathcal C}}

\def \Ac{{\mathcal A}}
\def \Ec{{\mathcal E}}
\def \Pc{{\mathcal P}}

\def \Hc{\mathcal{H}}
\def \Mc{{\mathcal M}}
\def \eps{\varepsilon}
\def \Leb{\operatorname{\texttt{Leb}}}
\def \cone{\operatorname{\text{cone}}}
\def \ul{\operatorname{\text{loc}}}
\def \ul{\operatorname{\textbf{ul}}}
\def \wg{\operatorname{\textbf{wg}}}

\title{Equilibrium under Time-Inconsistency: A New Existence Theory by Vanishing Entropy Regularization}

\author{ 
Zhenhua Wang \thanks{Zhongtai Securities Institute for Financial Studies, Shandong University, Jinan, Shandong, China. Email:\url{zhenhuaw@sdu.edu.cn}}
\and Xiang Yu\thanks{Department of Applied Mathematics,  The Hong Kong Polytechnic University, Kowloon, Hong Kong. Email:\url{xiang.yu@polyu.edu.hk}}
\and Jingjie Zhang\thanks{China School of Banking and Finance, University of International Business and Economics, Beijing, China. Email:\url{jingjie.zhang@uibe.edu.cn}}
	\and Zhou Zhou\thanks{School of Mathematics and Statistics, University of Sydney, Sydney, Australia. Email:\url{zhou.zhou@sydney.edu.au}} }

\begin{document}
\date{}
\maketitle	

\begin{abstract}
This paper develops a framework for establishing the existence of solutions to the equilibrium Hamilton–Jacobi–Bellman (EHJB) equation arising in time-inconsistent stochastic control problems. The time-inconsistency in our setting arises from the initial-time dependence such as the non-exponential discounting. The classical approach typically relates the existence of equilibrium to the classical solution of the EHJB, whose existence is still an open problem under general model assumptions. We resolve this challenge by building on a vanishing entropy regularization approach. Using fixed-point arguments, we first establish the existence of classical solutions to the exploratory equilibrium Hamilton–Jacobi–Bellman Equation (EEHJB) by deriving a series of delicate PDE estimates for the solution and its derivatives. Building on these estimates for the solution of the EEHJB and its derivatives, we then conduct a rigorous convergence analysis under suitable norms as the entropy regularization vanishes. Our main result shows that solutions of the EEHJB converge to a strong solution of the original EHJB, corresponding to the limit of the regularized equilibria. This convergence yields a verification argument ensuring that the limiting relaxed equilibrium indeed constitutes an equilibrium for the original time-inconsistent control problem. We thus establish the well-posedness of the EHJB and the existence of equilibria in diffusion models under time-inconsistency, without resorting to conventional stringent regularity assumptions of the EHJB.      

\end{abstract}

\maketitle

\section{Introduction}
Vast real-life financial and economic problems are compromised to be time-inconsistent due to the non-exponential discounting adopted by the decision maker, that is, a policy deemed optimal today may no longer remain optimal at future dates. In response to the failure of global optimality, a more suitable alternative solution, initiated by \cite{Strotz}, is to look for the subgame perfect Nash equilibrium for the intra-personal game among the decision maker's current and future selves. In the continuous-time setting, the seminal study of \cite{Bjork17} considered a perturbation of the control for an infinitesimal period of time and used the limit of time-discretization to derive the extended HJB equation as a system of nonlinear PDEs. By proving a rigorous verification theorem, the equilibrium can be characterized by the classical solution of the extended HJB equation in \cite{Bjork17}. From the approximation perspective by discrete-time intra-personal games, \cite{Yong2012} formulated a nonlinear and nonlocal PDE, called the equilibrium HJB equation (EHJB), which is mathematically equivalent to the extended HJB system in \cite{Bjork17} after some transformations. These two pioneer studies inspired a wave of interesting work on continuous-time time-inconsistent control problems, primarily focusing on the analysis of extended HJB system or EHJB in different models. As the verification theorem plays a vital role in establishing the existence and characterization of equilibrium, the conventional approach in the literature stands on the stringent regularity conditions of the classical solution to the extended HJB system or EHJB. However, analyzing the existence of classical solution to the general system of nonlinear and nonlocal PDEs remains a challenging open problem. Recently, \cite{lei_nonlocal_2023} investigated the short-time classical solution to the general EHJB under specific model assumptions and \cite{lei_nonlocal_2024} further established the existence of global solution of EHJB for time-inconsistent game problems relying on stronger conditions on model coefficients and cost functions. When these restrictive model assumptions are violated or when the existence of classical solution to extended HJB or EHJB is unknown, equilibrium in continuous-time models under time-inconsistency remains nearly unexplored.  

In the context of continuous-time reinforcement learning (RL), the entropy regularization was first introduced by \cite{wang2020reinforcement} to encourage the exploration during the learning procedure by its induced action randomization as relaxed controls. In the linear-quadratic setting, the optimal policy can be explicitly characterized as a Gaussian distribution, a special case of the Gibbs measure, which degenerates to the optimal deterministic policy in the classical formulation as the regularization tends to zero. Since then, entropy regularization has attracted an upsurge of interest in various RL problems, see \cite{wang2020continuous},
\cite{jia2022policy_gradient_continuous_time}, \cite{guo2022entropy}, \cite{jia2023q}, \cite{dai2023learning}, \cite{dong}, \cite{DFX24}, \cite{dai2024learning}, \cite{bo2025optimal}, \cite{pham25}, \cite{wei2025continuous}, \cite{HLYZ25}, among others. 

However, due to the presence of the additional entropy term in the objective function, it has been well noted that the exploratory stochastic control problem in fact deviates from the original problem without the entropy, which raises an important question in the exploratory formulation: will the exploratory formulation converge to the original one as the entropy regularization vanishes? In the time-consistent setting, this problem has been fully addressed in \cite{tang2022exploratory} by hinging on the well-developed stability theory of the viscosity solution of the exploratory HJB equation with respect to the temperature parameter. Recently, in the same realm of viscosity solution theory and stability analysis, the convergence of the system of exploratory HJB equations has also been established by \cite{HLYZ25} for optimal regime-switching problems.

Apart from the scope of RL, it has also been unfolded in some recent studies that the Gibbs measure characterization of the optimal relaxed control under entropy regularization brings mathematical simplifications in exercising fixed point arguments. By further establishing the convergence as entropy regularization vanishes, it offers a new route to confirm the existence of solution in the original problem as the limit from the regularized problems. Such an approach was initiated in \cite{bayraktar2025relaxed}, which studied time-inconsistent MDP problems under general discounting and established the existence of relaxed equilibrium in both discrete-time and continuous-time settings using the approach of vanishing entropy regularization; Subsequently, \cite{YuYuan26} investigated the time-inconsistent mean-field MDP of stopping and proved that the regularized equilibrium converges to the relaxed equilibrium; In discrete-time major-minor mean-field games of stopping, \cite{YZZZ} employed the vanishing entropy regularization approach in the major-player's problem and proved the existence of mean-field equilibrium; \cite{DZ24} studied the convergence of the exploratory extended HJB equations for continuous-time mean-variance stopping problems by assuming the existence of classical solutions, which leads to the existence of equilibrium in the original problem.          

Inspired by some observed technical convenience of entropy regularization in aforementioned studies, this paper aims to revisit the time-inconsistent control problems when the drift coefficient of state dynamics is controlled and develop a new theory for the existence of equilibrium by analyzing the convergence of PDEs under vanishing entropy regularization. On one hand, unlike the discrete-time models in \cite{bayraktar2025relaxed}, \cite{YuYuan26}, \cite{YZZZ}, we have to shift the focus from the convergence of operators to the more delicate convergence analysis of associated HJB equations arising from the continuous-time setting. Some technical estimations of the solution and its derivatives are inevitable. On the other hand, in sharp contrast to  \cite{tang2022exploratory} and \cite{HLYZ25} for classical time-consistent optimal control problems, we can no longer resort to some stability results of viscosity solutions in the literature. Indeed, the dynamic programming principle fails in our model and we are in the absence of necessary  viscosity characterization of the equilibrium value function. The rigorous convergence analysis and the verification of equilibrium in our general setting call for new technical tools and PDE analysis. Comparing with the recent work \cite{DZ24}, we do not assume any priory LQ solution structure and do not impose the restrictive assumption on the existence of classical solution to the EHJB for the proof of convergence analysis.

Our methodology and theoretical contributions are outlined as follows:

Firstly, we consider the time-inconsistent stochastic control problem under the Shannon entropy regularization and derive the exploratory equilibrium HJB equation (EEHJB) system \eqref{eq:cha.HJBentropy}--\eqref{eq:cha.HJBentropy'}. The main contribution of this step is establishing the existence of classical solution to the EEHJB system by invoking the fixed point argument on a specialized compact space with the help of the Gibbs-form policy operator in \eqref{eq:cha.HJBpientropy}. Thanks to the sublinear growth estimate of the entropy term in Lemma \ref{lm:entropy.ln}, we first derive H\"older norm estimates for the Gibbs-form operator  $\pi(x,a):= \Gamma_{\lambda}(x, D_x w(x), a)$ with an initial input $w(x)\in \Cc^{1 }_{\alpha}(\R^d)$ (see Lemma \ref{lm:pi.est}). We then show that the associated value function $V_{\lambda}^{\pi}$ satisfies the EEHJB system and enjoys certain weak type of H\"older regularity estimates (see Lemma \ref{prop:iteration.estimateentropy}). Building upon these estimates, the equilibrium essentially boils down to the fixed point of the operator $\Phi_\lambda(w):=V^{\Gamma_\lambda(D_x w(0,x))}$, for which a key step is to identify a suitable compact set. Our first main  contribution is to propose to work with $\mathcal{M}_{\lambda}$ defined by \eqref{compact-set} as a specialized compact subset of $\Cc^{1,2, \wg}_{\beta/2, \beta}(\T\times \R^d)$ together with a tailor-made weighted global H\"{o}lder norm. This construction ensures that $\Phi_\lambda$ defines a continuous self-mapping on $\mathcal{M}_{\lambda}$ (see Lemma \ref{lm:Phi.continuous}). Thanks to these preparations, Schauder fixed-point theorem can be called such that the value function under the fixed point satisfies the EEHJB system with the desired regularity (see Theorem \ref{thm:equil.existenceentropy}), which enables us to conclude the existence of regularized equilibrium by some standard verification arguments.

Secondly, we cope with the more challenging task by analyzing the convergence of EEHJB system towards the original EHJB system when the entropy parameter $\lambda\rightarrow 0$. To this end, we consider a sequence of pairs $(v^n,\pi^n)_{n\in \N}$ satisfying the nonlinear PDEs \eqref{eq:HJBn} and \eqref{eq:HJBn'} in the EEHJB system, respectively. Building on the norm estimates established in the previous section and employing a diagonal argument, we first show the existence of a subsequence $v^{n_k}$ converging locally on $D_N$ to a function  $v^\infty\in \Cc^{0,1 }_{\alpha/2, \alpha}(\T\times \R^d)\cap  W^{1,2, \ul}_{p}(\T\times \R^d)$ under the H\"{o}lder norm $\Cc^{0,1 }_{\alpha/2, \alpha}(D_N)$ for each $N\in\mathbb{N}$, whose derivatives also converge to the derivative of $v^{\infty}$ in the distribution sense on each local domain $D_N$ (see item (i) of Lemma \ref{lm:thm.C12andyoung}). Along the same subsequence, we establish the weak convergence of $\pi^{n_k}$ towards a Borel measurable $\pi^{\infty}$ with the aid of Young measure theory. As the second main contribution, we address the most intricate step to relate the limit $v^{\infty}$ to a strong solution to the EHJB equation in the original model under the limit $\pi^{\infty}$ such that $v^{\infty}=V^{\pi^{\infty}}$, see Theorem \ref{thm:equi.existence}. We achieve this core convergence result  by developing novel arguments to show that  $v^\infty$ is locally a strong solution of the EHJB and $v^\infty$ can be extended from the local limit on $D_N$ to the entire domain (see item (iii) of Lemma  \ref{lm:thm.C12andyoung}). As a consequence of this convergence result, we establish the existence of a strong solution to the nonlinear EHJB system in Theorem \ref{thm:equi.existence} for the original time-inconsistent control problem, a first kind of result in the literature. Moreover, our convergence analysis can also be interpreted as the stability result of the solution to the EEHJB system with respect to $\lambda\rightarrow 0$ in the time-inconsistent setting, which has not been explored in the literature.

Furthermore, our final goal is to show that the limit $\pi^\infty$ constitutes a relaxed equilibrium in the original problem, for which we highlight a new technical contribution: verification arguments based on the previous convergence analysis.  In fact, the limit $v^\infty\in \Cc^{0,1 }_{\alpha/2, \alpha}(\T\times \R^d)\cap  W^{1,2, \ul}_{p}(\T\times \R^d)$ is identified as a strong, rather than classical, solution to the original EHJB system. However,  thanks to the established convergence of $v^{n_k}$ to $v^\infty$ and the fact that each $v^{n_k}$ is a classical solution of the EEHJB, we take full advantage of the convergence and the It\^{o}-Krylov formula, together with localization arguments and estimates around $t=0$, to conclude that $\pi^{\infty}$ indeed satisfies the definition of relaxed equilibrium in \eqref{eq:def.equi} (see Theorem \ref{thm:equi.existence1}). Unlike most existing studies, our verification proof does not lean on the classical solution. Instead, we establish a new and weaker sufficient condition in Theorem \ref{thm:equi.existence1} for the existence
and characterization of equilibrium, which only requires a strong solution in $\Cc^{0,1 }_{\alpha/2, \alpha}(\T\times \R^d)\cap  W^{1,2, \ul}_{p}(\T\times \R^d)$ to the original EHJB. Notably,  Theorem \ref{thm:equi.existence} guarantees via a convergence argument that such a strong solution to the EHJB system always exists, thereby establishing the general existence of equilibrium.  Therefore, our new existence theory (Theorem \ref{thm:equi.existence1} and Theorem \ref{thm:equi.existence} ) opens a new avenue to study time-inconsistent control problems when the strong regularity conditions (such as classical solution) of the EHJB are difficult to prove or simply unavailable.    

Meanwhile, it is worth mentioning a recent study \cite{HYZ26} that established the policy iteration convergence for time-inconsistent control problems under entropy regularization, which partially laid the theoretical foundation for the design of some policy-iteration-based RL algorithms in time-inconsistent settings. Similar to the time-consistent counterpart, one natural question arises: although for a fixed entropy regularization, the greedy iteration of the Gibbs-form policies can lead to the regularized equilibrium, can the learned solution in the exploratory formulation approximate the true equilibrium when the temperature parameter is set to be small? As a byproduct of our main results, we provide an affirmative answer to this question by showing that the regularized equilibrium converges to a relaxed equilibrium of the original problem. This finding offers theoretical justification for the use of small temperature parameters in reinforcement learning algorithms, even in time-inconsistent settings. Moreover, in contrast to \cite{HYZ26}, where the existence of classical solutions to the EEHJB equation is established via a constructive policy iteration argument restricted to the finite-horizon setting, we develop a fundamentally different approach based on direct fixed-point arguments. 
%Our method not only yields the existence of classical solutions to the EEHJB equation but also extends naturally to the finite-horizon case. 
In addition, the new PDE estimates developed in our analysis play a crucial role in the subsequent convergence study as the entropy regularization vanishes.

The rest of the paper is organized as follows. Section \ref{sec:model} introduces the formulation of time-inconsistent control problems, the definition of relaxed equilibrium and the associated EHJB system. Section \ref{sec:entropy} formulates the problem under entropy regularization and derives the induced EEHJB system. The existence of a classical solution to the EEHJB system is established therein, together with carefully tailored norm estimates that play a crucial role in the subsequent convergence analysis. Section \ref{sec:convergence} carries out the convergence analysis of the solutions to the EEHJB towards a strong solution of the EHJB,  and developed new verification arguments to confirm that the limit of regularized equilibrium is the equilibrium in the original time-inconsistent control problem as entropy regularization vanishes.

\subsection{Notations} Let $\N$ be the set of all positive integers and denote the infinite time horizon as $\T:=[0,\infty)$. In the $m$-dimensional Euclidean space $\R^m$, we denote by $| \cdot|$ the Euclidean norm, by $\Leb(\cdot)$ the Lebesgue measure, and by $B_r(x)$ the ball centered at $x$ with radius $r$.

Given $w(t,x):\T\times{\R^d}\to {\R}$, let $\partial_tf$ denote the right derivative on the time variable $t$ and $D_xw$ (resp. $D^2_xw$) denote the first (resp. second) partial derivative vector (Hessian) on the space variable $x$. We first introduce the notation for H\"older spaces. Given a simple connected domain $D\subset \T\times {\R^d}$, define
$$
\begin{aligned}
	&\|w\|_{\Cc^0(D)}:=\sup_{(t,x)\in D} |w(t,x)|= \|w\|_{L^\infty(D)},\quad 
	[w]_{\Cc_{\alpha/2, \alpha}(D)}:=\sup_{\substack{(t,x)\neq (s,y)\in D \\ (|t-s|+|x-y|^2)\leq 1}}\frac{|w(t,x)-w(s,y)|}{(|t-s|+|x-y|^2)^{\alpha/2}}, \\
	& \|w\|_{\Cc_{\alpha/2, \alpha}(D)}:= [w]_{\Cc_{\alpha/2, \alpha}(D)}+\|w\|_{\Cc^0(D)}.
\end{aligned}
$$
Given a multi-index $a=(a_1,\cdots, a_d)$ with $|a|_{l_1}:= \sum_{i=1}^d a_i$, define $D_x^a w:= \frac{\partial^{|a|_{l_1}} w }{\partial^{a_1}_{x_1}\cdots \partial^{a_d}_{x_d}}$
and we say $w\in \Cc^{l, k}_{\alpha/2,\alpha}(D)$ 
%(resp. $w\in \Cc^{1, 2}(D)$) 
if 
$$
\|w\|_{\Cc^{l, k}_{\alpha/2,\alpha}(D)}:= \sum_{0\leq j\leq l, 0\leq |a|_{l_1}\leq k} \|\partial^j_t D^a_x w \|_{\Cc_{\alpha/2, \alpha}(D)}<\infty. %\quad \left(\text{resp. $\|w\|_{\Cc^{1, 2}(D)}:= \sum_{2l+k\leq 2} \|\partial^l_t D^k_xw \|_{\Cc^0(D)}<\infty$} \right).
$$
The H\"older seminorms are also applied when restricted to either the spatial or time variable alone, e.g., $[f(\cdot)]_{\alpha(O)}= \sup_{\substack{x\neq y\in O \\ |x-y|\leq 1}}\frac{|f(x)-f(y)|}{|x-y|^{\alpha}}$ for $f$ defined on $O\subset \R^d$, and $[g(\cdot)]_{\alpha/2(O)}= \sup_{\substack{t\neq s\in O \\ |t-s|\leq 1}}\frac{|g(t)-g(s)|}{|t-s|^{\alpha/2}}$ for $g$ defined on $O\subset \T$.

Define $D_N(t_0, x_0):= (t_0, t_0+N)\times B_N(x_0)$  and write $D_N:= D_N(0,0)$, for $N>0$. 
We further define the following weighted global H\"older norm on the whole space $\T\times \R^d$
$$
\|w\|_{\Cc^{1, k, \wg}_{\alpha/2,\alpha}(\T\times \R^d)}:=\sum_{N\in \N} \frac{1}{2^N} \|w\|_{\Cc^{1, k}_{\alpha/2,\alpha}(D_N)},
$$
%uniformly local H\"older norm and %\|w\|_{\Cc^{1, k }_{\alpha/2,\alpha}(\T\times \R^d)}:=\sup_{(t,x)\in \T\times \R^d}  \|w\|_{\Cc^{1, k}_{\alpha/2,\alpha}([t,t+1]\times B_1(x))},\quad   
and we denote by $\Cc^{1, k}_{\alpha/2,\alpha}(D)$ (resp. $\Cc^{1, k, \wg}_{\alpha/2,\alpha}(\T\times \R^d)$) the normed space of all functions with finite norm $\| \cdot \|_{\Cc^{1, k}_{\alpha/2,\alpha}(D)}$ (resp. finite norm $\|\cdot\|_{\Cc^{1, k, \wg}_{\alpha/2,\alpha}(\T\times \R^d)}$). %Notice that $\Cc^{1, k}_{\alpha/2,\alpha}(D)$ is a Banach space for any suitable domain, and $\Cc^{1, k, \wg}_{\alpha/2,\alpha}(\T\times \R^d)$ is not due to the fact that $\Cc^{1, k, \wg}_{\alpha/2,\alpha}(\T\times \R^d)$ is not completed.

%Given a simple connected bounded domain $D\subset \T\times {\R^d}$, 
We also use the Sobolev norm $\|w\|_{W^{1,k}_p(D)}:=  \sum_{0\leq l\leq 1, 0\leq |a|_{l_1}\leq k} \|\partial^l_t D^a_x w \|_{L^p(D)}$, and we further define the uniformly local Sobolev norm on the whole space $\T\times \R^d$
% and weighted global Sobolev norm on the whole space
$$
\begin{aligned}
	\|w\|_{ W^{1,2, \ul}_{p} (\T\times \R^d)}:=& \sum_{0\leq l\leq 1, 0\leq |a|_{l_1}\leq k} \left(\sup_{(t,x)\in \T\times \R^d}   \|\partial^l_t D^a_x w \|_{L^p(D_1(t,x))} \right)
	%	\|w\|_{W^{1,2, \wg}_{p} (\T\times \R^d)}:=&    \sum_{N=1}^\infty \frac{1}{2^N}  \|w\|_{W^{1, 2}_p(D_N)}.
\end{aligned}
%\left(\sup_{(t,x)\in \T\times \R^d}   \|\partial^l_t D^a_x w \|_{L^p([t,t+1]\times B_1(x))} \right)
$$

Throughout the paper, let $0<\alpha<1$ be an arbitrarily fixed H\"older constant, and we shall let the power $p$ in any Sobolev norm equal to $\frac{d+2}{1-\alpha}$. However, the whole arguments are still valid for other H\"older constants and other possible powers.

\section{Model Setup}\label{sec:model}
We consider a time-inconsistent stochastic control problem. Let $U$ denote the action space, which is a compact subset of $\R^\ell$ with $\Leb(U)>0$. We also denote by $\Pc(U)$ as the set of all probability measures on $U$ equipped with the weak topology, and $\Pc_c(U)$ as the subset of $\Pc(U)$ which contains all probability measures absolutely continuous with respect to the Lebesgue measure. For technical convenience, we focus on the model over an infinite horizon where only the drift coefficient of the state process is controlled. The general case when the diffusion coefficient is also controlled will be left for future study. 

To ensure the existence of solution to the general time-inconsistent control problem, we shall consider the relaxed equilibrium (or mixed-strategy equilibrium). To this end, let us consider an adapted relaxed control $\pi=(\pi_t)_{t\in \T}$, taking values in $\Pc(U)$, which is said to be admissible if the controlled SDE 
$$
dX^\pi_t=\int_U b(X^\pi_t, a)\pi_t(da) dt +\sigma(X^\pi_t) dW_t, X_0=x.
$$
admits a unique strong solution, where $b:\R^d\times U\to \R^d$, and $\sigma:\R^d\to \R^{d\times m}$ and $W$ is a standard $m$-dimensional Brownian motion. 

The value function under the relaxed control $\pi$ over an infinite horizon is defined by
$$
V^{\pi}(t,x)=\E_x\left[\int_0^\infty  \int_{U} r(t+s, X^\pi_s,a) \pi_s(da) ds \right],
$$
where $r(t,x,a): \T\times \R^d\times U\to \R$ is the reward function.
Let $J^{\pi}(x):=V^{\pi}(0,x)$ be the value function at the initial time $t=0$ under the relaxed control $\pi$. 

For the rest of the paper, we use notation
$$
f^\varpi(t,x):= \int_U f(t,x,a)\varpi(da)
$$
for any generic function $f(t, x, a): \T\times {\R^d}\times U\to \R$ and any generic distribution $\varpi\in \Pc(U)$.
And we use the shorthand $f^\pi(x)=f^{\pi(x)}(t,x)$ for a feedback control $\pi: \R^d\to \Pc(U)$.

\begin{Assumption}\label{assume.r}
	We assume that
	\be\label{eq:assume.integralr} 
	\begin{aligned}
		K_0:=&\sup_{a\in U}\|r(\cdot, a)\|_{\Cc^{1,0}_{\alpha/2, \alpha}(\T\times \R^d)}<\infty,\\
		K_1:=&\int_0^\infty  \sup_{x\in \R^d,a\in U}|r(t,x, a)| dt+ \int_0^\infty  \sup_{x\in \R^d,a\in U}|r_t(t,x, a)| dt < \infty,
	\end{aligned}
	\ee 
	and
	\be\label{eq:assume.rholder}    
	K_2:=\int_0^\infty  \sup_{x\in \R^d,a\in U, s< u\leq s+1}\frac{|\partial_t r(s,x, a)-\partial_t r(u,x, a)|}{|s-u|^{\alpha/2}} ds <\infty.
	\ee 
\end{Assumption}	

\begin{Remark}
	The boundedness condition in \eqref{eq:assume.integralr} ensures that $J^{\pi}$ is well-defined for any admissible relaxed control $\pi$, and the condition in \eqref{eq:assume.rholder} is satisfied when the second derivative of $r$ with respect to $t$ is integrable over the time horizon $\T$, uniformly in $a$ and $x$. This assumption will be used later to ensure the H\"older continuity of $\partial_t V^{\pi}$ for an admissible relaxed control $\pi$.
\end{Remark}

\begin{Assumption}\label{assume.b.sigma}
	Assume $K_3:=\sup_{a\in U}\|b(x, a)\|_{\Cc_{\alpha}(\R^d)}+\|\sigma(x)\|_{\Cc_\alpha(\R^d)}<\infty$, and there exists a constant $0<\eta<\infty$ such that
	%$\sigma$ satisfies the following uniform elliptic condition 
	\be\label{eq:asume.elliptic} 
	\eta|\xi|^2\leq \xi \sigma\sigma^T(x) \xi^T,\quad 
	%\leq \eta_1 |\xi|^2\quad  \
	\forall \xi\neq 0 \text{ and }x \in \R^d.
	\ee 
\end{Assumption}	

The goal of the agent is to maximize $J^{\pi}$ over all admissible relaxed controls. The non-exponential discounting renders the stochastic control problem time inconsistent. Since global optimality is no longer attainable, we seek a subgame perfect Nash equilibrium for the intra-personal game under which future selves have no incentive to deviate from the prescribed strategy. This game-theoretic perspective motivates the following definition.

\begin{Definition}\label{def:equi.relaxed}
	An admissible relaxed control $\pi^*$ is called an equilibrium if for any $x\in \R^d$
	\begin{equation}\label{eq:def.equi}
		\limsup_{\eps\to 0+} \frac{J^{\pi'\otimes_{\eps} \pi^*}(x)- J^{\pi^*}(x)}{\eps}\leq 0,\quad \forall \text{ admissible } \pi'.
	\end{equation}
\end{Definition}

This paper aims to study the existence of equilibrium by employing the vanishing entropy regularization approach. Such an approach is first proposed in showing the general existence of equilibrium for time-inconsistent  MDP problems in \cite{bayraktar2025relaxed}. In the existing literature, the existence of equilibrium is typically addressed by resorting to the existence of classical solution to the extended HJB or EHJB and some rigorous verification arguments, see Proposition \ref{prop:classic.HJB} below, whose proof is omitted because it is analogous to that of Theorem \ref{thm:equil.existenceentropy} Part (i).

\begin{Proposition}\label{prop:classic.HJB}
	Let Assumptions \ref{assume.r} and \ref{assume.b.sigma} hold. Suppose that there exists a classic solution $u(t,x)\in \Cc^{1,2}(\T\times \R^d)$ to the following nonlinear and nonlocal PDE system
	\begin{align}
		0=&\partial_t u(0,x)+\frac{1}{2}\tr\left(\sigma\sigma^T(x) D^2_x u(0,x)\right) 
		+\sup_{\varpi\in \Pc(U)} \left\{ \int_U \left[ b(x,a)D_x u(0,x) +r(0,x,a) \right] \varpi(da) \right\}, \label{eq:cha.HJB}\\
		0=&\partial_t u (t,x)+\frac{1}{2}\tr\left((\sigma\sigma^T)(x) D^2_x u (t,x)\right) +b^{\pi^*}(x)D_x u(t,x) +r^{\pi^*}(t,x) \label{eq:cha.HJB'},
	\end{align}
	where the relaxed control $\pi^*: \R^d\to \Pc(U)$ achieves the supremum of the equation \eqref{eq:cha.HJB}. Then we have that $\pi^*$ is an equilibrium and $V^{\pi^*}=u$. We call \eqref{eq:cha.HJB}–\eqref{eq:cha.HJB'} the equilibrium HJB (EHJB) system.
	%Conversely, given an admissible relaxed control $\pi^*: \R^d\mapsto \Pc(U)$, suppose the system \eqref{eq:cha.HJB}--\eqref{eq:cha.HJB'} admits a solution $u(t,x)$ with $u\in \Cc^{1,2}_{\alpha/2, \alpha}(\T\times \R^d)\cap W^{1,2 }_{p, \ul}(\T\times \R^d)$, then $\pi^*$ is an equilibrium and $V^{\pi^*}=u$. Moreover, $\pi^*$ satisfies \eqref{eq:cha.HJBpi} a.e. on $\R^d$.
\end{Proposition}

\begin{Remark}
	%\eqref{eq:cha.HJB}--\eqref{eq:cha.HJB'} (in the strong sense) together is commonly treated as the equilibrum HJB system which shall be satisfied by a value function of an equilibrium $\pi^*$.
	It is well known that proving the existence of a classical solution to the nonlinear and nonlocal EHJB system \eqref{eq:cha.HJB}–\eqref{eq:cha.HJB'} is extremely challenging. In the existing literature, classical solutions have only been constructed in some specific concrete models, typically via a guess-and-verify approach. 
	
	%In the time-consistent setting, i.e., $r(t,x,a)=e^{-\rho} f(x, a)$, the EHJB system \eqref{eq:cha.HJB}--\eqref{eq:cha.HJB'} simplifies to the classic HJB equation
	%$$
	%0=-\rho u(x)+\frac{1}{2}\tr(\sigma\sigma^T(x) D^2_x u(x))  +\sup_{\varpi\in \Pc(U)} \left\{ \int_U \left[ b(x,a) D_x u(x)+f(x,a) \right] \varpi(da) \right\}
	%$$
	%which admits a unique classic solution. In the contrast, one do not expect that the solution to the extend-HJB system \eqref{eq:cha.HJB}--\eqref{eq:cha.HJB'} is unique, because \eqref{eq:cha.HJB} is more like an initial condition.
	
	In contrast, one of our main findings is that searching for a classical solution is unnecessary. Indeed, the vanishing entropy regularization approach yields the existence of a strong solution (rather than a classical one) to the EHJB system, for which verification arguments remain applicable, thereby establishing the existence of an equilibrium; see Theorems \ref{thm:equi.existence} and \ref{thm:equi.existence1}.
	
\end{Remark}

\section{Existence of Equilibrium under Entropy Regularization}\label{sec:entropy}
In this section, let us first study the time-inconsistent control problem under Shannon entropy regularization on relaxed controls. That is, the regularized value function under an admissible relaxed control $\pi \in  \Pc_c(U)$ is now given by
\be\label{eq:def.Ventropy} 
V^{\pi}_\lambda(t,x)=\E_x\left[\int_0^\infty  \left( r^\pi(t+s, X^\pi_s)+\delta(t+s)\lambda \Hc(\pi_s)\right)ds \right],\quad J^{\pi}_\lambda(x)=V^{\pi}_\lambda(0,x),
\ee 
where $\Hc$ stands for the Shannon entropy: $\Hc(\varpi):= -\int_U \ln(\varpi(a))\varpi(a)da$ for a given density $\varpi\in \Pc_c(U)$. For technical convenience, we impose the following integrability assumptions on the general discounting function $\delta(t)$.
\be\label{eq:assume.delta} 
K_4:=\int_0^\infty (|\delta(t)|+|\delta_t(t)|)dt<\infty,
\ee 
and 
\be\label{eq:assume.deltaholder} 
K_5:=\int_0^\infty \sup_{s< u\leq s+1}\frac{|\delta_t(s)-\delta_t(u)|}{|s-u|^{\alpha/2}}ds<\infty.
\ee 
Note that \eqref{eq:assume.integralr} in Assumption \ref{assume.r}, together with \eqref{eq:assume.delta}, ensures that $V^{\pi}_\lambda$ is well-defined for all admissible policies $\pi \in  \Pc_c(U)$.

\begin{Definition}\label{def:equi.entropy}
	An admissible relaxed control $\pi^*$ is called a regularized equilibrium with the entropy parameter $\lambda$ if for any $x\in \R^d$,
	\begin{equation}\label{eq:def.equientropy}
		\limsup_{\eps\to 0+} \frac{J_\lambda^{\pi'\otimes_{\eps} \pi^*}(x)- J_\lambda^{\pi^*}(x)}{\eps}\leq 0,\quad \forall \text{ admissible }\pi'.
	\end{equation}
\end{Definition}

To address the general existence of classical solution to the EEHJB system \eqref{eq:cha.HJBentropy}--\eqref{eq:cha.HJBentropy'}, we need one more assumption that provides regularity of parameters on action variable $a$ and the action space $U$. To properly state the assumption, we illustrate the following notation for a cone in $\R^\ell$. When $\ell>1$, for any $\gamma\in [0,\pi/2]$, denote 
$$
\Delta_\gamma:=\left\{ a=(a_1,\cdots, a_{\ell})\in \R^\ell: \sum_{i=1}^{\ell-1}u_i^2\leq \tan(\gamma)a_{\ell}^2 \right\}
$$
as the cone with vertex, axis, and angle being $0\in \R^\ell$,  $a_1 =\cdots =a_{\ell-1} =0$, and $\gamma$, respectively, and we write the cone obtained by a rotation of $a+\Delta_\gamma$ in $\R^\ell$ about $a\in \R^{\ell}$ as $\cone(a,\gamma)$.
\begin{Assumption}\label{assume.lipsa.U}
	The functions $a\mapsto b(x,a)$ and $a\mapsto r(t,x,a)$ are uniformly Lipchitz, i.e., 
	$$
	\Theta:=\sup_{(t,x)\in \T\times \R^d} \frac{|b(x,a)-b(x,a')|+|r(t,x, a)|-|r(t,x,a')|}{|a-a'|}<\infty.
	$$
	
	The action space $U$ satisfies a uniform inner-cone test condition: When $\ell>1$, there exists $\zeta$ and $\gamma\in (0,\pi/2]$ such that for any $a\in U$, there is a cone with vertex $a$ and angle $\gamma$ (i.e., $\cone(a,\gamma)$) that satisfies $\cone(a,\gamma)\cap B_\zeta(a)\subset U$. When $\ell=1$, there exists $\zeta>0$ such that for any $a\in U$, either $[a-\zeta, a]$ or $[a, a+\zeta]$ is contained in $U$.
\end{Assumption}

\begin{Remark}
	We stress that the cone test condition is by no means restrictive. Indeed, $U$ being a convex space is a common assumption in the literature to guarantee the well-posedness of stochastic control problem, i.e., to make sure a regular optimal solution exists. Our cone test condition in Assumption \ref{assume.lipsa.U} is strictly weaker than $U$ being convex.
	
	%	{\color{blue}countable is not true, delete this line?}We also want to emphasize that all our analysis work, hence, our results hold for $U$ being finite or countable.
\end{Remark}

\begin{Theorem}\label{thm:equil.existenceentropy}
	Assume  Assumptions \ref{assume.r}, \ref{assume.b.sigma}, \ref{assume.lipsa.U}, \eqref{eq:assume.delta} and \eqref{eq:assume.deltaholder} hold. Then the followings hold.
	\bi
	\item[(i)]	
	If a feedback relaxed control $\pi^*: \R^d\to \Pc_c(U)$ is a regularized equilibrium with entropy parameter $\lambda$ and $V^{\pi^*}_\lambda \in \Cc^{1,2 }_{\alpha/2, \alpha}(\T\times \R^d)$, then $V^{\pi^*}$ solves the following exploratory equilibrium HJB (EEHJB) system
	\begin{align}
		0=& \partial_t u(0,x) +\frac{1}{2}\tr(\sigma\sigma^T(x) D^2_x u(0,x))\notag\\
		&+\sup_{\varpi\in \Pc_c(U)} \left\{ \int_U \left[ b(x,a) D_xu(0,x)+r(0,x,a)-\lambda \ln(\varpi(a)) \right] \varpi(a)da \right\}\label{eq:cha.HJBentropy}  \\
		=& \partial_t u_t(0,x)+\frac{1}{2}\tr(\sigma\sigma^T(x) D^2_x u(0,x) \notag\\
		& +\lambda \ln \left\{ \int_U \exp\left(\frac{1}{\lambda}\left[ b(x,a) D_x u(0,x)+r(0,x,a)\right] \right) \right\} , \forall x \in \R^d;   \notag \\
		0=&\partial_t u(t,x)+\frac{1}{2}\tr(\sigma\sigma^T(x) D^2_x u(t,x)) \notag \\
		&+b^{\pi^*}(x) D_x u(t,x)+r^{\pi^*}(t,x)+\lambda \delta(t) \Hc(\pi^*(x)), \forall (t,x) \in \T\times \R^d  \label{eq:cha.HJBentropy'}, 
	\end{align}
	where it follows from \eqref{eq:cha.HJBentropy} that
	\be\label{eq:cha.HJBpientropy} 
	\pi^*(x,a)=\frac{\exp\left( \frac{1}{\lambda}\left[b(x,a) \cdot D_x u(0,x) +r(0,x,a) \right]\right)  }{\int_U \exp\left( \frac{1}{\lambda}\left[b(x,a')\cdot D_x u(0,x)+r(0,x,a') \right]\right)  da'}  \quad \forall x\in \R, a \in U.
	\ee
	
	\item[(ii)]	Conversely, if a couple $(u,\pi^*)$ satisfies that $u(t,x)\in \Cc^{1,2 }_{\alpha/2, \alpha}(\T\times \R^d)$ being a classic solution to the EEHJB system \eqref{eq:cha.HJBentropy}--\eqref{eq:cha.HJBentropy'} and \eqref{eq:cha.HJBpientropy} holds. Then, $\pi^*$ is a regularized equilibrium with entropy weight $\lambda$ and $u(t,x)=V^{\pi^*}_\lambda(t,x)$. 
	
	\item[(iii)] There exists $\lambda_0>0$ such that for all entropy parameters $0<\lambda\leq \lambda_0$, there exists a pair $(u,\pi^*)$ with $u(t,x)\in \Cc^{1,2}_{\alpha/2,\alpha}(\T\times\R^d)$ solving the EEHJB system \eqref{eq:cha.HJBentropy}--\eqref{eq:cha.HJBentropy'}, and $\pi^*$ satisfies \eqref{eq:cha.HJBpientropy}. As a consequence, $u(t,x)=V^{\pi^*}_\lambda(t,x)$, and $\pi^*$ is an equilibrium under entropy weight $\lambda$. Moreover, the following estimates also hold:
	\be\label{eq:thm.entropy} 
	\begin{aligned}
		&\|V^{\pi^*}_\lambda\|_{\Cc^{0,1 }_{\alpha/2, \alpha}([t,\infty)\times \R^d)}\vee \|V^{\pi^*}_\lambda \|_{W^{1,2, \ul}_{p}([t,\infty)\times \R^d)}\leq A^* \Psi(t)\quad \forall t\in \T;\\
		&\sup_{x\in \R^d}\|V^{\pi^*}_\lambda(\cdot, x)\|_{\Cc^1_{\alpha/2}(\T)}\vee \|V^{\pi^*}_\lambda\|_{\Cc^{1,2 }_{\alpha/2, \alpha}(\T \times \R^d)}\leq A_\lambda.
	\end{aligned}
	\ee
	Here $A^*$ is a constant depending on $K_0,\cdots, K_5$, $\eta$, $\Theta, \zeta, \gamma$ in Assumptions \ref{assume.r}, \ref{assume.b.sigma}, \ref{assume.lipsa.U}, and conditions \eqref{eq:assume.delta}--\eqref{eq:assume.deltaholder}, but independent of $\lambda$, and $A_\lambda$ is another constant that also depends on $\lambda$. $\Psi$ is a deterministic function from $ \T$ to $\T$ that satisfies $\Psi(0)=1$, $\Psi$ decreases and $\lim_{t\to\infty} \Psi(t) = 0$.
	\ei 
\end{Theorem}

\subsection{Proof for Theorem \ref{thm:equil.existenceentropy}}

We begin by proving Part (i) of Theorem \ref{thm:equil.existenceentropy}. The proofs of Parts (ii) and (iii) will be presented after establishing several auxiliary lemmas. In particular, the proof of Part (ii) will follow Lemma \ref{prop:iteration.estimateentropy}, while the proof of Part (iii) will be delivered after Lemma \ref{lm:Phi.continuous}.

\begin{proof}[{\bf Proof for Theorem \ref{thm:equil.existenceentropy} Part (i): }] Suppose $\pi^*: \R^d\to \Pc_c(U)$ is a regularized equilibrium such that $V^{\pi^*}_\lambda\in \Cc^{1,2 }_{\alpha/2, \alpha}(\T\times \R^d)$. We will show that \eqref{eq:cha.HJBentropy}--\eqref{eq:cha.HJBpientropy} hold.
	
	Let us first verify \eqref{eq:cha.HJBentropy'}. For any $(t,x)\in \T\times \R^d$, we have that, for any $\eps>0$,
	$$
	E_x\left[\int_0^\eps \left[ r^{\pi^*}(t+s, X^{\pi^*}_s)+\delta(t+s)\lambda\Hc(\pi^*(X^{\pi^*}_s))\right]ds+V^{\pi^*}_\lambda(t+\eps, X^{\pi^*}_\eps)\right]-V^{\pi^*}_\lambda(t,x)=0.
	$$
	In light that $V^{\pi^*}_\lambda\in \Cc^{1,2}_{\alpha/2, \alpha} (\T\times \R^d)$, the classical It\^o formula yields
	$$
	\begin{aligned}
		0=\frac{1}{\eps}\E_x\bigg[&\int_0^\eps \bigg( \partial_t V^{\pi^*}_\lambda(t+s, X^{\pi^*})+b^{\pi^*}(X^{\pi^*}) D_xV^{\pi^*}_\lambda(t+s, X^{\pi^*}_s) \\ &\qquad +\frac{1}{2}\tr(\sigma\sigma^T(X^{\pi^*}_s) D^2_x V^{\pi^*}_\lambda(t+s, X^{\pi^*}_s))+r^{\pi^*}(t+s, X^{\pi^*}_s) + \delta(t+s)\lambda\Hc(\pi^*(X^{\pi^*}_s)) \bigg)ds\bigg].
	\end{aligned}
	$$
	In view of  Assumptions \ref{assume.b.sigma} and \ref{assume.r} and the fact that $V^{\pi^*}_\lambda\in \Cc^{1,2 }_{\alpha/2, \alpha}(\T\times \R^d)$ again, the right-hand side above converges to 
	$$
	\partial_t V^{\pi^*}_\lambda(t,x)+\frac{1}{2}\tr\left((\sigma\sigma^T)(x) D^2_x V^{\pi^*}_\lambda(t,x)\right) +  b^{\pi^*}(x)D_xV^{\pi^*}_\lambda(t,x)+r^{\pi^*}(t,x) +\delta(t)\lambda\Hc(\pi^*(x)),
	$$
	which shows that \eqref{eq:cha.HJBentropy'} holds under $\pi^*$. 
	
	Next, we deal with \eqref{eq:cha.HJBentropy}.  Taking $t=0$ in \eqref{eq:cha.HJBentropy'} yields
	\be\label{eq:prop.1'}
	\begin{aligned}
		0=&\partial_t V^{\pi^*}_\lambda(0,x)+\frac{1}{2}\tr\left((\sigma\sigma^T)(x)  D^2_x V^{\pi^*}_\lambda(0,x)\right) +  b^{\pi^*}(x)D_x V^{\pi^*}_\lambda(0,x)+r^{\pi^*}(0,x) +\lambda\Hc(\pi^*(x))\\
		\leq &\partial_t V^{\pi^*}_\lambda(0,x)+\frac{1}{2}\tr\left((\sigma\sigma^T)(x)  D^2_x V^{\pi^*}_\lambda(0,x)\right) \\
		& +\sup_{\varpi\in \Pc_c(U)} \left\{ \int_U \left[ b(x,a) D_x V^{\pi^*}_\lambda(0,x)+r(0,x,a) -\lambda\ln(\varpi(a))\right] \varpi(a)da \right\}.
	\end{aligned}
	\ee
	Take an arbitrary $\varpi\in \Pc_c(U)$ with $|\Hc(\varpi)|<\infty$ %\footnote{Here $\varpi$ is understood as a constant policy} 
	and $x\in \R^d$.
	%and the value function $V^{\pi'\otimes_\eps {\pi^*}}(t,x)$ is well-defined, 
	For any $\eps>0$, it holds that
	$$
	J_\lambda^{\varpi\otimes_{\eps}\pi^*}(x)=V^{\varpi\otimes_{\eps}\pi^*}_\lambda(0,x)=\E\left[\int_0^\eps\left[   r^{\varpi}(s, X^{\varpi}_s) +\delta(s)\lambda\Hc(\varpi)) \right]ds +V^{\pi^*}_\lambda\left( \eps, X^{\varpi}_\eps \right)\right].
	$$
	It follows that
	\be\label{eq:prop.0} 
	\begin{aligned}
		&\frac{J^{\varpi\otimes_{\eps}\pi^*}_\lambda(x)-J^{\pi^*}_\lambda(x)}{\eps}=\\
		& \frac{1}{\eps} \left\{\E_x\left[\int_0^\eps \left( r^{\varpi}(s, X^{\varpi}_s)+\delta(s)\lambda\Hc(\varpi)  \right)ds\right] +\E_x\left[V^{\pi^*}_\lambda(\eps, X^{\varpi}_\eps)\right]-V^{\pi^*}_\lambda(0,x) \right\}.
	\end{aligned}
	\ee 
	Again applying It\^o formula to $V^{\pi^*}_\lambda\in\Cc^{1,2}_{\alpha/2, \alpha}(\T\times \R^d)$, we obtain that the limit of the right-hand side as $\eps\to 0$ is
	\be\label{eq:prop.0'}
	\begin{aligned}
		&\partial_t V^{\pi^*}_\lambda(0,x)+\frac{1}{2}\tr\left( \sigma\sigma^T(x)D_x^2V^{\pi^*}_\lambda(0,x) \right)\\
		&+\int_U\left[ b(x,a)D_xV^{\pi^*}_\lambda(0,x)+r(0,x, a)-\lambda\ln(\varpi(a)) \right] \varpi(d)da.
	\end{aligned}
	\ee
	By the definition of regularized equilibrium, we conclude that 
	\bee
	\begin{aligned}
		&\partial_t V^{\pi^*}_\lambda(0,x)+\frac{1}{2}\tr\left(\sigma\sigma^T(x) D^2_x V^{\pi^*}_\lambda(0,x)\right) \\ &+\sup_{\varpi\in \Pc_c(U)} \left\{ \int_U \left[ b(x,a) D_x V^{\pi^*}_\lambda(0,x)+r(0,x,a) -\lambda\ln(\varpi(a)) \right] \varpi(da) \right\}\leq 0.
	\end{aligned}
	\eee 
	This, together with \eqref{eq:prop.1'}, yields that $V^{\pi^*}_\lambda(0,\cdot)$ satisfies \eqref{eq:cha.HJBentropy}. Then, \eqref{eq:cha.HJBpientropy} is a direct consequence of \eqref{eq:cha.HJBentropy} and \eqref{eq:cha.HJBentropy'}.\\
\end{proof}	

As stated in Theorem \ref{thm:equil.existenceentropy} Part (i), a regularized equilibrium with the entropy parameter $\lambda>0$  is a density functional on $D_xV^{\pi^*}(0,x)$ in Gibbs form. For any fixed $\lambda>0$, define $\Gamma_{\lambda}: \R^d\times \R^d\to \Pc_c(U)$ by
\be\label{eq:def.Gammalambda}
\begin{aligned}
	\Gamma_{\lambda}(x,p,a):=& \argmax_{\varpi\in \Pc_c(U)}\left\{ \int_U \left[b(x,a)p+r(0,x,a) -\lambda \ln(\varpi(a))\right]\varpi(a)da \right\}\\
	=& \frac{\exp\left( \frac{1}{\lambda}\left[b(x,a)p +r(0,x,a) \right]\right)}{\int_U \exp\left( \frac{1}{\lambda}\left[b(x,a')p+r(0,x,a') \right]\right) da'}.
\end{aligned}
\ee 
%Then \eqref{eq:cha.HJBpientropy} gives that $\pi^*(x, a)=\Gamma_\lambda(x,D_x V^{\pi^*}(0,x), a)$ if $\pi^*$ is a regularized equilibrium.

The following lemma is borrowed from \cite[Lemma 1]{bayraktar2025relaxed} (also see  \cite[Corollary 4.14]{2025SICON}), which provides a sublinear growth estimate for the entropy term on the Gibbs-form operator $\Gamma_\lambda(x,p,a)$ in terms of $|p|$.
\begin{Lemma}\label{lm:entropy.ln}
	Suppose that Assumption \ref{assume.lipsa.U} holds. There exists a constant $\lambda_0>0$ such that the following estimate holds for all $0<\lambda\leq \lambda_0$,
	\bee
	|\Hc(\Gamma_\lambda(x,p,a))|=\left| \int_U \ln(\Gamma_{\lambda}(x,p,a))\Gamma_{\lambda}(x,p,a)da\right|\leq K_6+K_7|\ln\lambda|+K_8 \ln(1+|p|),\quad \forall x, p\in \R^d,
	\eee
	where $K_6, K_7, K_8$ are positive constants depending on $\ell$, $\Leb(U)$ and the constants $\Theta$, $\zeta$ and $\gamma$ in Assumption \ref{assume.lipsa.U}, but independent of $\lambda$.
\end{Lemma}

The following lemma establishes  H\"older norm estimates for key quantities induced by the policy
\be\label{eq:def.gammaw} 
\pi(x,a):= \Gamma_{\lambda}(x, D_x w(x), a),
\ee 
for a given function $w(x)\in \Cc^{1 }_{\alpha}(\R^d)$. 
\begin{Lemma}\label{lm:pi.est}
	Suppose Assumptions \ref{assume.r}, \ref{assume.b.sigma} hold. If $w(x)\in \Cc^{1 }_{\alpha}(\R^d)$ satisfies $\|w\|_{\Cc^{1}_\alpha(\R^d)}\leq M$ for a finite constant $M\geq 1$,  then the following estimates hold for the policy $\pi$ in \eqref{eq:def.Gammalambda} :
	$$
	\begin{aligned}
		&\|\lambda \Hc(\pi)\|_{\Cc_{\alpha} (\R^d)}\leq A_0 M \exp(\frac{A_0M}{\lambda}), \quad  \|b^{\pi}\|_{\Cc_{\alpha} (\R^d)} \vee \|r^{\pi}\|_{\Cc_{\alpha/2, \alpha} (\T\times \R^d)}\leq \frac{A_0 M}{\lambda}\exp(\frac{A_0M}{\lambda}),
	\end{aligned}
	$$
	where $A_0$ is a constant depending on $d$, $\Leb(U)$ and the parameters $K_0, K_1, K_3$ in Assumptions \ref{assume.r} and \ref{assume.b.sigma}, but independent of $\lambda\leq \lambda_0$ and $M$.
\end{Lemma}

\begin{proof}
	Throughout this proof, let $|x-y|\leq 1$ and $A_0$ be a generic positive constant depending on $d$, $\Leb(U)$ and $K_0, K_1, K_3$ in Assumptions \ref{assume.r} and \ref{assume.b.sigma} but independent of $\lambda$ and $M$, which may vary from line to line.
	
	Let $g(x, a):=\frac{1}{\lambda}[b(x,a)D_x w(x)+r(0,x,a)]$. A direct calculation gives that
	\be\label{eq:lm.piest.0} 
	\sup_{a\in U}\|g(\cdot,a)\|_{\Cc_{\alpha}(\R^d)}\leq \frac{A_0 M}{\lambda},
	\ee 
	which implies that
	\begin{equation}\label{eq:lm.piest.1}   
		\exp(\frac{-A_0 M}{\lambda})  \leq \pi(x,a) \leq \exp(\frac{A_0 M}{\lambda}).
	\end{equation}
	This further yields that 
	\be\label{eq:lm.piest.3}   
	\sup_{a\in U}[\exp(g(\cdot, a))]_{\Cc_{\alpha}(\R^{d})}\leq \frac{A_0 M}{\lambda}\exp(\frac{A_0 M}{\lambda}).
	\ee 
	Note that 
	$$
	\begin{aligned}
		&|\pi(x,a)-\pi(y,a)|\\
		&=\left| \frac{  \exp[g(x,a)]\int_U \exp[g(y,a')]da' -   \exp[g(y,a)]\int_U \exp[g(x,a')]da' }{   \int_U \exp[g(x,a')]da'\int_U \exp[g(y,a')]da' } \right|\\
		&\leq \left|  \frac{ \int_U \left( \exp [g(y,a')]-\exp[g(x, a')]\right)da'  }{\int_U \exp[g(y,a')]da'} \right|   \pi(x,a)+  \left|  \frac{ \exp[g(x,a)]-\exp[g(y, a)]  }{\int_U \exp[g(y,a')]da'} \right|.
	\end{aligned}
	$$
	This, together with \eqref{eq:lm.piest.1} and \eqref{eq:lm.piest.3}, yields that
	\be\label{eq:lm.piest.5}     
	\sup_{a\in U}[\pi(\cdot, a)]_{\Cc_{\alpha}(\R^d)}\leq \frac{A_0 M}{\lambda}\exp(\frac{A_0M}{\lambda})
	\ee 
	By \eqref{eq:lm.piest.1} again and the fact that $\pi(x, a)$ is a probability density on $U$ for each $x$, we deduce that 
	$$
	\left|\int_U \ln(\pi(x,a))\pi(x,a)da \right|\leq \frac{A_0 M}{\lambda}.
	$$
	Also, it holds that
	\be\label{eq:lm.piest.5'} 
	\begin{aligned}
		&\left| \int_U \ln(\pi(x,a))\pi(x, a)da  - \int_U \ln(\pi(y,a))\pi(y, a)da  \right|\\
		&\leq \sup_{a\in U} \left|   \ln(\pi(x,a))  -  \ln(\pi(y,a)) \right| +\sup_{a\in U} \left| \frac{\pi(y,a)}{\pi(x,a)} - 1 \right| |\ln(\pi(y,a))|.
	\end{aligned}
	\ee 
	In view of \eqref{eq:lm.piest.0}--\eqref{eq:lm.piest.3}, the first term on the right-hand side above is estimated by
	$$
	\begin{aligned}
		&\left|   \ln(\pi(x,a))  -  \ln(\pi(y,a)) \right|  \\
		&\leq\left| g(x,a)-g(y,a)\right|+ \left| \ln \left( \int_U \exp[g(x, a)]da \right) -\ln \left( \int_U \exp[g(y,a)]da \right) \right|
		\leq \frac{A_0 M}{\lambda} |x-y|^{\alpha}.
	\end{aligned}
	$$
	To estimate the second term on the right-hand side of \eqref{eq:lm.piest.5'}, let us assume without loss of generality that $\pi(y,a)\geq \pi(x,a)$. Then, 
	$$
	\begin{aligned}
		\left| \frac{\pi(y,a)}{\pi(x,a)} - 1 \right|  \leq & \frac{\exp[g(y,a)]\int_U \exp\left[g(y,a')  + [g(\cdot, a)]_{\alpha} |x-y|^{\alpha} \right] da' }{  \int_U \exp[g(y,a')]da' \exp[g(x,a)]}  - 1\\
		= & \exp[g(y,a)-g(x,a)]\exp\left[  \sup_{a\in U} [g(\cdot,a)]_{\Cc_{\alpha}(\R^d)} |x-y|^{\alpha}   \right]-1\\
		\leq & \frac{A_0 M}{\lambda}\exp(\frac{A_0M}{\lambda}) |x-y|^{\alpha}.
	\end{aligned}
	$$
	Combining above with \eqref{eq:lm.piest.1} and \eqref{eq:lm.piest.5'}, we conclude that 
	$$
	\|\lambda \Hc(\pi)\|_{\Cc_{\alpha} (\R^d)}\leq A_0 M \exp(\frac{A_0M}{\lambda}).
	$$
	By virtue of Assumptions \ref{assume.b.sigma} and \ref{assume.r}, \eqref{eq:lm.piest.1} and\eqref{eq:lm.piest.5}, one can get
	$$
	\|b^{\pi}\|_{\Cc_{\alpha} (\R^d)}, \|r^{\pi}\|_{\Cc_{\alpha/2, \alpha} (\T\times \R^d)}\leq \frac{A_0 M}{\lambda}\exp(\frac{A_0M}{\lambda}).
	$$
\end{proof}

\begin{Lemma}\label{prop:iteration.estimateentropy}
	Let Assumptions \ref{assume.r}, \ref{assume.b.sigma}, \ref{assume.lipsa.U}, and conditions \eqref{eq:assume.delta}--\eqref{eq:assume.deltaholder} hold. Let $ \lambda_0$ be given by Lemma \ref{lm:entropy.ln} and fix $\lambda \in (0, \lambda_0]$.  For a given $w(x)\in \Cc^{1 }_{\alpha}(\R^d)$, consider the feedback control 
	$$\pi(x,a):= \Gamma_{\lambda}(x, D_x w(x), a).$$
	Then $V^{\pi}_\lambda(t,x)$ defined in \eqref{eq:def.Ventropy} is the unique classic solution to the following PDE
	\be\label{eq:prop.iteraPDE}  
	\partial_t u(t,x)+  \frac{1}{2}\tr\left((\sigma\sigma^T)(x) D^2_x u(t,x)\right)+ b^{\pi}(x)D_x u(t,x) +r^{\pi}(t,x)+\lambda \delta(t)\Hc(\pi(x))=0, \,\, \forall (t,x) \in \T \times \R^d.
	\ee 
	And there exists a function $\Psi: \T\to \T$ and constants $A^*, A_\lambda$, satisfying:
	\begin{itemize}
		\item $\Psi(0)=1$, and $\Psi$ decreases with $\lim_{t\to\infty} \Psi(t) = 0$,
		\item $A^*$ depends only on $d, \ell$, $\alpha$, $\Leb(U)$ and the parameters $K_0,\cdots,  K_5$, $\eta$, $\Theta, \zeta, \gamma$ in Assumptions \ref{assume.r}, \ref{assume.b.sigma}, \ref{assume.lipsa.U}, and conditions \eqref{eq:assume.delta}--\eqref{eq:assume.deltaholder}, but independent of $\lambda\leq \lambda_0$; $A_\lambda$ further depends on $\lambda$,
	\end{itemize}
	such that if $\|w\|_{\Cc^{1 }_\alpha (\R^d)}\leq A^*$, then the following estimate holds: 
	\begin{align}
		&\|V^{\pi}_\lambda\|_{\Cc^{0,1 }_{\alpha/2, \alpha}([t,\infty)\times \R^d)}\vee \|V^{\pi}_\lambda \|_{W^{1,2, \ul}_{p}([t,\infty)\times \R^d)}\leq A^* \Psi(t)\quad \forall t\in \T \label{eq:prop.est.t};\\
		&\sup_{x\in \R^d}\|V^{\pi}_\lambda(\cdot, x)\|_{\Cc^1_{\alpha/2}(\T)}\leq A^*,\quad  \|V^{\pi}_\lambda\|_{\Cc^{1,2 }_{\alpha/2, \alpha}(\T\times \R^d)}\leq A_\lambda \label{eq:prop.est}.
	\end{align}
\end{Lemma}

\begin{proof}
	%Fix $0<\lambda\leq \lambda_0$. Take $w\in \Cc^{1,2}_{\alpha/2, \alpha}(\T\times \R^d)$, and set $\pi(x,a):=\Gamma_{\lambda}(x, D_x w(0,x))$. 
	Throughout the proof, let $A_0$ be a generic positive constant depending on $d, \alpha$ and $K_1,\cdots, K_5$, $\eta$, $\Theta, \zeta, \gamma$ in Assumptions \ref{assume.r}, \ref{assume.b.sigma}, \ref{assume.lipsa.U} and conditions \eqref{eq:assume.delta}--\eqref{eq:assume.deltaholder}, but independent of $\lambda$ and norms of all orders of derivatives of $w$, which may vary from line to line.  
	
	{\bf Step 1.} We first provide estimates on the solution to \eqref{eq:prop.iteraPDE}. Let $u(t,x)$ be a classic solution to \eqref{eq:prop.iteraPDE}. By Lemma \ref{lm:entropy.ln}, 
	\be\label{eq:prop.est1} 
	|\lambda\Hc(\pi(x))|\leq A_0\left( 1+\ln\left(1+\|D_x w\|_{\Cc^0(\R^d)} \right) \right).
	\ee
	Define 
	\be\label{eq:lm.est5'}  
	\Psi(t):= \frac{ \sup_{x\in \R^d, a\in U}\|r(\cdot,x,a)\|_{L^\infty([t,\infty))} \vee \|\delta\|_{L^\infty([t,\infty))}}{\sup_{x\in \R^d, a\in U}\|r(\cdot,x,a)\|_{L^\infty(\T)} \vee \|\delta\|_{L^\infty(\T)}},
	\ee
	then $\Psi(0)=1$, $\Psi(t)$ decreases, and we conclude from Assumption \ref{assume.r} and \eqref{eq:assume.delta} that 
	\be\label{eq:lm.est5}    
	\lim_{t\to\infty} \Psi(t)=0.
	\ee
	Recall that $p= \frac{d+2}{1-\alpha}$. For an arbitrary fixed $(t_0, x_0)\in \T\times \R^d$, Sobolev embedding (see, e.g., \cite[II.3]{ladyzhenskaia1968linear}) gives that 
	\be\label{eq:lm.embed} 
	\|u\|_{\Cc^{0,1}_{\alpha/2 ,\alpha}(D_1(t_0, x_0))}\leq A_0 \|u\|_{W^{1, 2}_p(D_2(t_0, x_0))}.
	\ee 
	It follows from Sobolev estimate (see, e.g., \cite[Chapter 2, Section 4, Lemma 4]{krylov2008lectures}) that
	\bee
	\begin{aligned}
		\|u\|_{W_p^{1,2}(D_2(t_0, x_0))}
		&\leq A_0\bigg( \|r^\pi\|_{L^\infty(D_3(t_0,x_0))} +\|\lambda\delta\Hc(\pi)\|_{L^\infty(D_3(t_0, x_0))}  \bigg).
	\end{aligned}
	\eee
	By combining this with \eqref{eq:prop.est1} \eqref{eq:lm.est5'} and \eqref{eq:lm.embed}, we conclude for any $t\in \T$, %and \eqref{eq:prop.tnorm},
	$$
	\begin{aligned}
		&\sup_{(t_0,x_0)\in [t,\infty)\times \R^d} \left(\|u\|_{W^{1,2}_{p}(D_1(t_0, x_0))}\vee\|u \|_{\Cc^{0,1}_{\alpha/2,\alpha}(D_1(t_0, x_0))} \right)\\
		&\leq A_0 \left(\sup_{a\in U}\|r(\cdot,a)\|_{[t,\infty)\times \R^d} \vee \|\delta\|_{L^\infty([t,\infty))} \right) \left(1+\ln\left(1+\|D_x w\|_{\Cc^0(\R^d)} \right)\right).\\
		&\leq A_0 \left(\sup_{a\in U}\|r(\cdot,a)\|_{\T\times \R^d} \vee \|\delta\|_{L^\infty(\T)} \right) \Psi(t) \left(1+\ln\left(1+\|D_x w\|_{\Cc^0(\R^d)} \right)\right). %\quad \text{uniformly over $(t_0, x_0)\in [t,\infty)$},
	\end{aligned}
	$$
	This together with \eqref{eq:assume.integralr} and \eqref{eq:assume.delta} yields, for any $t\in \T$ that,
	\be\label{eq:lm.est7}   
	\|u\|_{W^{1,2,\ul}_{p}([t,\infty)\times \R^d)}\vee\|u \|_{\Cc^{0,1}_{\alpha/2,\alpha}([t,\infty)\times \R^d)} \leq A_0 \Psi(t) \left( 1+\ln \left( 1+\|w\|_{\Cc^1(\R^d)} \right)\right).
	\ee

	{\bf Step 2.} Next, we show $V^{\pi}_\lambda$ is the unique classic solution to \eqref{eq:prop.iteraPDE}. Since $\|w\|_{\Cc^1_\alpha(\R^d)}<\infty$, by Assumptions \ref{assume.r}, \ref{assume.b.sigma}, \eqref{eq:assume.delta} and Lemma \ref{lm:pi.est}, we get that 
	$$
	\begin{aligned}
		&\|b^{\pi}\|_{\Cc^0_\alpha(\R^d)}, \;
		%\;  \| (\sigma^2)^{\pi(x)}(x)\|_{\Cc^0_\alpha(\R^d)} <\infty\\ 
		\|r^{\pi}\|_{\Cc_{\alpha/2, \alpha}(\T\times \R^d)},  \; \|\delta\lambda\Hc(\pi)\|_{\Cc_{\alpha/2, \alpha}(\T\times \R^d)}<\infty.
	\end{aligned}
	$$
	%the functions $x\mapsto b^{\pi(x)}(x), \sigma^{\pi(x)}(x)$, and the function $(t,x)\mapsto r^{\pi(x)}(t,x), \delta(t)\lambda\Hc(\pi(x))$ are globally bounded and H\"older continuous on $\R$, and on $\T\times \R$, respectively. 
	%Rewrite $b^{\pi(x)}(x)$, $r^{\pi(x)}(t, x)$ by $b^{\pi}(x)$, $r^{\pi}(t, x)$ respectively for short. 
	Then the SDE
	$
	dX^{\pi}_t= b^\pi(X^\pi_t)dt+\sigma(X^\pi_t)dW_t
	$
	admits a unique strong solution. A standard application of It\^o's formula yields that $V^{\pi}_\lambda(t,x)$ satisfies the parabolic PDE \eqref{eq:prop.iteraPDE}. Now suppose $u$ is a classic solution to \eqref{eq:prop.iteraPDE}, and we show $V^{\pi}_\lambda=u$. Take an arbitrary $(t,x)\in \T\times \R^d$. Applying It\^o- formula to $u(t+s, X^{\pi}_s)$ yields
	\begin{align*}
		du(t+s, X^{\pi}_s)
		= &\left(\partial_t u(t+s, X^{\pi}_s)+ b^{\pi}(x) D_x u(t+s,X^{\pi}_s)+\frac{1}{2}\tr\bigg((\sigma\sigma^T)(X^{\pi}_s) D^2_x u(t+s,X^{\pi}_s) \bigg)\right) ds\\
		& + \sigma(X^{\pi}_s)D_xu(X^{\pi}_s)dW_s.
	\end{align*}
	%Due to 
	%that $X^{\pi^*}_s$ has a local H\"older continuous density for all $s>0$, and 
	%that the $Y_s:=\int_0^s \sigma(X^{\pi^*})D_x u(t+r, X^{\pi^*}_r)dW_r$ is a martingale, 
	Taking expectations on both sides leads to
	$$
	\begin{aligned}
		u(t,x)=& \E_x\bigg[  \int_0^T \left(u_t(t+s, X^{\pi}_s)+ b^{\pi}(x) D_x u(t+s,X^{\pi}_s)+\frac{1}{2}\tr\left((\sigma\sigma^T)(X^{\pi}_s) D^2_x u(t+s,X^{\pi}_s)\right) \right) ds \\
		%	&+\E_x\left[ \int_0^T \sqrt{(\sigma^2)^{\pi^*}(X^{\pi^*}_s)}v^\infty_{x}(t+s, X^{\pi^*}_s)dW_s \right]\\
		&\qquad +u(t+T, X^{\pi^\infty}_T) \bigg]\\
		=& \E_x\left[  \int_0^T \left( r^{\pi}(t+s,X^{\pi}_s) +\delta(t+s)\lambda\Hc(\pi(X^\pi_s)) \right) ds \right]+\E_x\left[ u(t+T, X^{\pi}_T) \right]
	\end{aligned}
	$$
	where the second equality follows from \eqref{eq:prop.iteraPDE}. By \eqref{eq:lm.est5} and \eqref{eq:lm.est7}, 
	$\lim_{T\to\infty}\ E_x\left[u(t+T, X^{\pi}_T) \right] = 0$. Therefore, by \eqref{eq:assume.integralr} in Assumption \ref{assume.r} and the Dominated Convergence Theorem, we conclude that the last line above converges to
	$\E_x\left[  \int_0^\infty \left(r^{\pi}(t+s,X^{\pi}_s)+\delta(t+s)\lambda\Hc(\pi(X^\pi_s)) \right)  ds \right]$. Consequently, it holds that $V^{\pi^*}_\lambda(t,x)=u(t,x)$ on $\T\times \R^d$. 
	
	%Moreover, by standard Schauder estimate, $\|V^{\pi}_\lambda\|_{\Cc^{1,2}_{\alpha/2, \alpha}(\T\times \R^d)}<\infty$.

	{\bf Step 3.} We now prove \eqref{eq:prop.est.t} and \eqref{eq:prop.est}. By \eqref{lm:entropy.ln}, for each $(t,x)\in \T\times \R^d$, it holds that
	\begin{align}
		\|V^{\pi}_\lambda\|_{\Cc^0(\T\times \R^d)} 
		%=&\left|\E_x\left[\int_0^\infty \left(r^{\pi}(t+s, X^\pi_s)+\delta(t+s)\lambda \Hc(\pi(X^\pi_s)) \right) ds \right] \right
		\leq &\int_{0}^\infty \left[\sup_{(x, a)\in \R^d\times U}| r(s, x, a)|+ \delta(s)A_0 \bigg(1+\ln\big(1+\|D_xw\|_{\Cc^0(\R^d)} \big) \bigg) \right] ds \notag\\
		\leq &A_0 \big(1+\ln\left(1+\|D_x w\|_{\Cc^0(\R^d)} \right)\big), \label{eq:prop.V}\\
		|\partial_t V^{\pi}_\lambda(t,x)|
		=& \E_x\left[ \int_{0}^\infty \left(\partial_t  r^{\pi}(t+s, X^\pi_s)+ \delta_t(t+s)\lambda \Hc(\pi(X^\pi_s)) \right) ds  \right]  \notag\\
		\leq & \int_{0}^\infty \left[ \sup_{(x, a)\in \R^d\times U}|\partial_t  r(s, x, a)|+ \delta_t(s)A_0 \bigg( 1+\ln(1+\|D_x w\|_{\Cc^0(\R^d)} \big) \bigg)\right] ds\notag\\
		\leq &A_0 \big(1+\ln(1+\|D_x w\|_{\Cc^0(\R^d)})\big),  \label{eq:prop.Vt}
	\end{align}
	where the inequalities above follow from \eqref{eq:assume.integralr} and \eqref{eq:assume.delta}. Because $\pi$ is independent of $t$, we have, for $l=0,1$,  that
	\begin{align}
		&\sup_{t_0<t_1\leq t_0+1}\frac{\left|\partial^{l}_t V^{\pi}_\lambda(t_0,x) -\partial^{l}_t V^{\pi}_\lambda(t_1,x)\right|}{|t_0-t_1|^{\alpha/2}} \notag\\
		&\leq \sup_{t_0<t_1\leq t_0+1}\E_x\int_0^\infty \Bigg(\frac{\big|  \partial^l_t  r(t_0+s, X^\pi_s)-\partial^l_t  r(t_1+s, X^\pi_s)\big|}{|t_0-t_1|^{\alpha/2}} \notag\\
		&\qquad \qquad \qquad\qquad \qquad +\frac{\left|\partial^l_t \delta(t_0+s)-\partial^l_t \delta(t_1+s)  \right|}{|t_0-t_1|^{\alpha/2}} \lambda \Hc(\pi(X^\pi_s)) \Bigg) ds \label{eq:prop.5}\\
		&\leq \int_0^\infty \Bigg[\Bigg(\sup_{(y, a)\in \R^d\times U,t_0<t_1\leq t_0+1 }\frac{\left|\partial^l_t  r(t_0+s, y, a)-\partial^l_t  r(t_1+s, y, a)\right|}{|t_0-t_1|^{\alpha/2}} \notag\\
		&\qquad \qquad\quad +\sup_{t_0<t_1\leq t_0+1}\frac{\left|\partial_t^l \delta(t_0+s)-\partial_t^l \delta(t_1+s)\right|}{|t_0-t_1|^{\alpha/2}} \Bigg)A_0 \bigg(1+\ln\big(1+\|D_xw\|_{\Cc^0(\R^d)} \big) \bigg) \Bigg] ds \notag\\
		&\leq A_0 \left(1+\ln\left(1+\|D_x w\|_{\Cc^0(\R^d)} \right)\right),\notag
	\end{align}
	where the last inequality follows from \eqref{eq:assume.rholder} and \eqref{eq:assume.deltaholder}.
	Then, combining \eqref{eq:prop.V}--\eqref{eq:prop.Vt} leads to 
	\be\label{eq:prop.tnorm}  
	\sup_{x\in \R^d}\|V^{\pi}_\lambda(\cdot, x)\|_{\Cc^1_{\alpha/2}(\T)}\leq A_0 \left(1+\ln\left(1+\|D_x w\|_{\Cc^0(\R^d)} \right)\right).
	\ee 
	By Steps 1 and 2, 
	\be\label{eq:lm.est7'}   
	\|V^{\pi}_\lambda\|_{W^{1,2,\ul}_{p}([t,\infty)\times \R^d)}\vee\| V^{\pi}_\lambda \|_{\Cc^{0,1}_{\alpha/2,\alpha}([t,\infty)\times \R^d))} \leq A_0 \Psi(t) \left( 1+\ln \left( 1+\|w\|_{\Cc^1(\R^d)} \right)\right).
	\ee
	Observe that the function $y\mapsto A_0\left( 1+\ln \left( 1+y\right)\right)$, mapping from $[0,\infty)$ to $[0,\infty)$, has sub-linear growth. Consequently, there exists a constant $A^*$ such that 
	$$
	0\leq A_0\left( 1+\ln \left( 1+y\right)\right)\leq A^*, \text{ whenever } 0\leq y\leq A^*.
	$$ 
	We can thus conclude from \eqref{eq:prop.tnorm}  and \eqref{eq:lm.est7'}  that 
	\bee
	\begin{cases}
		\|V^{\pi}_\lambda\|_{W^{1,2,\ul}_{p}([t,\infty)\times \R^d)}\vee\|V^{\pi}_\lambda \|_{\Cc^{0,1}_{\alpha/2,\alpha}([t,\infty)\times \R^d))} \leq \Psi(t) A^*, \;\; \forall t\in\T \\ 
		\sup_{x\in \R^d}\|V^{\pi}_\lambda(\cdot, x)\|_{\Cc^1_{\alpha/2}(\T)}\leq  A^*
	\end{cases}
	\text{provided that $\|w\|_{\Cc^{1}_\alpha(\R^d)}\leq A^*$. }
	\eee
	Moreover, utilizing the condition on $\sigma$ in Assumption \ref{assume.b.sigma} and the interior Schauder estimate (see, e.g., \cite[Theorem 8.11.1 and Remark 8.11.2]{krylov1996lectures}), we deduce that, for any $(t_0, x_0)\in \T\times \R^d$,
	$$
	\begin{aligned}
		\|V^{\pi}_\lambda\|_{\Cc^{1,2}_{\alpha/2, \alpha}(D_1(t_0, x_0))}\leq & \tilde C \left( \|r^\pi\|_{\Cc_{\alpha/2,\alpha}(D_2(t_0, x_0))}+ \|\lambda \delta\Hc(\pi)\|_{\Cc_{\alpha/2,\alpha}(D_2(t_0, x_0))} \right),
	\end{aligned}
	$$
	where $\tilde C$ is a constant only depending on $\|b^\pi\|_{\Cc_{\alpha}(\R^d)}$ and the H\"older continuity of $\sigma$ such that $\tilde C$ can be bounded by $\varphi(\|b^\pi\|_{\Cc_{\alpha}(\R^d)})$ with a fixed increasing function $\varphi(\cdot):[0,\infty)\to [0,\infty)$. Therefore, by Lemma \ref{lm:pi.est}, \eqref{eq:assume.delta} and \eqref{eq:assume.deltaholder}, if  $\|w\|_{\Cc^{1}_\alpha(\R^d)}\leq A^*$, then it holds that
	$$
	\begin{aligned}
		&\|V^{\pi}_\lambda\|_{\Cc^{1,2}_{\alpha/2, \alpha}(D_1(t_0, x_0))}
		\leq  \varphi\left(\frac{A_0A^*}{\lambda}\exp(\frac{A_0A^*}{\lambda})\right)\frac{A_0A^*}{\lambda}\exp(\frac{A_0A^*}{\lambda}) =:A_\lambda \quad \forall (t_0, x_0)\in \T\times \R^d,
	\end{aligned}
	$$
	which leads to 
	$
	\|V^{\pi}_\lambda\|_{\Cc^{1,2}_{\alpha/2,\alpha}(\T\times \R^d)} \leq A_\lambda.
	$
\end{proof}

Now we are ready to show Part (ii) of Theorem \ref{thm:equil.existenceentropy}.
\begin{proof}[{\bf Proof for Theorem \ref{thm:equil.existenceentropy} Part (ii): }] 
	Let $(u,\pi^*)$ be a pair such that $u(t,x)\in \Cc^{1,2 }_{\alpha/2, \alpha}(\T\times \R^d)$ is a classical solution to the EEHJB system \eqref{eq:cha.HJBentropy}--\eqref{eq:cha.HJBentropy'} and \eqref{eq:cha.HJBpientropy} holds. 
	Lemma \ref{prop:iteration.estimateentropy} implies that, with $\pi^*(x,a)=\Gamma_\lambda(x,D_x u(0,x),a)$, the function $V^{\pi^*}_\lambda(t,x) \in \Cc^{1,2}_{\alpha/2,\alpha}(\T\times \R^d)$  is the unique solution to \eqref{eq:cha.HJBentropy'}. Thus, $V^{\pi^*}_\lambda=u$. It remains to show that \eqref{eq:def.equientropy} holds for $\pi^*$.
	%Lemma \ref{prop:iteration.estimateentropy} implies that $V^{\tilde\pi}_\lambda(t,x) \in \Cc^{1,2}_{\alpha/2,\alpha}(\T\times \R^d)$ with $\tilde\pi(x,a)=\Gamma_\lambda(x,D_x V^{\pi^*}(x),a)$. By \eqref{eq:cha.HJBpientropy}, $\tilde\pi=\pi^*$. Thus, $V^{\pi^*}_\lambda=u$. 
	For any $x\in \R^d$ and $\varpi\in \Pc_c(U)$, %It\^o formula leads to \eqref{eq:prop.0}. Then, 
	since $V^{\pi^*}_\lambda(t,x) \in \Cc^{1,2}_{\alpha/2,\alpha}(\T\times \R^d)$, by applying It\^o formula as in the proof for Part (i) of Theorem  \ref{thm:equil.existenceentropy}, we obtain equation  \eqref{eq:prop.0}, and taking the limit as $\lambda \to 0$ then yeilds \eqref{eq:prop.0'}, which is bounded by the right-hand side of \eqref{eq:cha.HJBentropy}. Therefore,  $\pi^*$ satisfies \eqref{eq:def.equientropy}  and hence is a regularized equilibrium.
	%From the regularity of $V^{\pi^*}$, the right-hand side of \eqref{eq:prop.0} equals the term of ``$\eqref{eq:prop.0'}\times \eps+o(\eps)$", which, by \eqref{eq:cha.HJBentropy},  is less or equal to $o(\eps)$. By the arbitrariness of $x$ and $\varpi$, we conclude that \eqref{eq:def.equientropy} holds for $\pi^*$, and thus, $\pi^*$ is a regularized equilibrium.
\end{proof}

Next, we shall use the fixed point argument to show the existence of a solution to the EEHJB system \eqref{eq:cha.HJBentropy}--\eqref{eq:cha.HJBentropy'}. To begin with, for any $w\in \Cc^{0,1}(\T\times \R^d)$, let us define $\Phi_\lambda(w)$ by
\be\label{eq:Phi.fixedpoint} 
\Phi_\lambda(w):=V^{\Gamma_\lambda(D_x w(0,x))}.
\ee 
Fix an arbitrary $0<\lambda\leq \lambda_0$ and define 
\begin{align}\label{compact-set}
	\Mc_\lambda:=\left\{ w\in \Cc^{1,2 }_{\alpha/2, \alpha}(\T\times \R^d) : \|w\|_{\Cc^{0,1 }_{\alpha/2, \alpha}(\T\times \R^d)}\leq A^*, \|w\|_{\Cc^{1,2}_{\alpha/2, \alpha}(\T\times \R^d)}\leq A_\lambda \right\}.
\end{align}
Note that $\Mc_\lambda$ is a convex subset of $\Cc^{1,2}_{\alpha/2, \alpha}(\T\times \R^d)$.

\begin{Lemma}\label{lm:Phi.prepare}
	Take a sequence $(w^n)_{n\in \N\cup\{\infty\}}\subset \Mc_\lambda$ with $\lim_{n\to\infty}\|w^n-w^\infty\|_{\Cc^{1,2,\wg}_{\alpha/2,\alpha}(\T\times \R^d)} = 0$. Set $\pi^n(x,a)=\Gamma_\lambda(x,D_xw^n(0,x),a)$ for all $n\in \N\cup\{\infty\}$. Then for an each $N\in \N$, it holds that
	$$
	\lim_{n\to\infty}\left(\|r^{\pi^n}-r^{\pi^\infty}\|_{\Cc_{\alpha/2,\alpha}(D_{N})}+ \|b^{\pi^n}-b^{\pi^\infty}\|_{\Cc_{\alpha}(B_{N}(0))}+\|\lambda\delta (\Hc(\pi^n)-\Hc(\pi^\infty))\|_{\Cc_{\alpha/2, \alpha}(D_{N})}\right)=0.
	$$
\end{Lemma}

\begin{proof}
	
	Through the proof,  let $A_0$ be a generic finite constant depending on $A^*, A_\lambda$, $N_0,\lambda$, $\Leb(U)$ but independent of $n$, which may change from line to line. Write $g^n(x,a):=\frac{1}{\lambda}[b(x,a)D_x w^n(0,x,a)+r(0,x,a)]$ for $n\in \N\cup\{\infty\}$. Note that 
	\be\label{eq:lm.C12est0'}     
	\sup_{a\in U}\|g^n(\cdot, a)\|_{\Cc_\alpha(B_{N_0+1} (0))}\leq A_0, \quad\sup_{a\in U}\|\exp(g^n(\cdot,a))\|_{\Cc_\alpha(B_{N_0+1} (0))}  \leq A_0,\quad \forall n\in \N\cup\{\infty\}.
	\ee
	Given some genetic functions $f(x), g(x), h(x,a)$ with $D$ being a domain in $\R^d$, the following results hold 
	\be\label{eq:lm.C12est0}    
	[fg]_{\Cc_\alpha(D)}\leq \|f\|_{L^\infty(D)}[g]_{\Cc_\alpha(D)}+ \|g\|_{L^\infty(D)}[f]_{\Cc_\alpha(D)}, \;  \left[ \int_U h(x,a)da \right]_{\Cc_\alpha(D)}\leq \sup_{a\in U} [h(\cdot,a)]_{\Cc_\alpha(D)}.
	\ee 
	First, \eqref{eq:lm.C12est0'} and \eqref{eq:lm.C12est0} together imply that
	\be\label{eq:lm.C12est0''}    
	\sup_{a\in U}\| (g^n-g^\infty)(\cdot, a)\|_{\Cc_{\alpha}(B_{N_0+1}(0))}\leq A_0 \| D_x(w^n-w^\infty)(0,x)\|_{\Cc_{\alpha}(B_{N_0+1}(0))}.
	\ee
	Direct computations yield that
	$$
	\begin{aligned}
		&\ln\left( \int_U \exp[g^n(x,a)] da \right) -\ln\left( \int_U \exp[g^\infty(x,a)]da \right) \\
		&= \ln\left(  \int_U  \frac{\exp[g^n(x,a)-g^\infty(x,a)]\exp[g^\infty(x,a)]}{\int_U \exp[g^\infty(x,a')]da'} da\right) \\
		&= \ln\left(  \int_U \exp[g^n(x,a)-g^\infty(x,a)] \pi^\infty(x,a)da \right).
	\end{aligned}
	$$
	Applying \eqref{eq:lm.C12est0} to the integral inside the logarithm and combining it with \eqref{eq:lm.C12est0'} and \eqref{eq:lm.C12est0''}, we have 
	\be\label{eq:lm.C12est7'} 
	\left\| \ln\left( \frac{\int_U \exp[g^n(x,a)] da} {\int_U \exp[g^\infty(x,a)]da} \right)  \right\|_{\Cc_\alpha(B_{N_0+1})(0)}\leq A_0 \| D_x(w^n-w^\infty)(0,x)\|_{\Cc_{\alpha}(B_{N_0+1}(0))}.
	\ee 
	In view that
	$$
	\exp[g^n(x,a)]- \exp [g^\infty(x,a)] =\exp[g^\infty(x,a)] \left(   \exp[g^n(x,a)-g^\infty(x,a)]-1] \right), 
	$$
	applying \eqref{eq:lm.C12est0} to the above equality, together with \eqref{eq:lm.C12est0'}, yields
	\bee
	\begin{aligned} 
		\left[e^{g^n(x,a)}- e^{g^\infty(x,a)} \right]_{\Cc_{\alpha}(B_{N_0+1}(0))}&\leq A_0  \sup_{a\in U} \bigg([e^{g^\infty(\cdot,a)}]_{\Cc_\alpha (B_{N_0+1}(0))} \|(g^n-g^\infty) (\cdot, a)\|_{L^\infty(B_{N_0+1}(0))} \\
		&\qquad\qquad \quad  +  \|(g^\infty) (\cdot, a)\|_{L^\infty(B_{N_0+1}(0))} [(g^n-g^\infty)(\cdot,a)]_{\Cc_\alpha(B_{N_0+1}(0))} \bigg),
	\end{aligned}
	\eee 
	which further implies that 
	\be\label{eq:lm.C12est1}   
	\|\exp[g^n(x,a)]- \exp [g^\infty(x,a)] \|_{\Cc_\alpha(B_{N_0+1}(0))}\leq A_0 \|D_x(w^n-w^\infty)(0,\cdot)\|_{\Cc_{\alpha}(B_{N_0+1}(0))}.
	\ee 
	Note that 
	$$
	\begin{aligned}
		&\pi^n(x,a)-\pi^\infty(x,a)\\
		&= \frac{  \exp[g^n(x,a)]\int_U \exp[g^\infty(x,a')]da' -   \exp[g^\infty(x,a)]\int_U \exp[g^n(x,a')]da' }{   \int_U \exp[g^n(x,a')]da'\int_U \exp[g^\infty(x,a')]da' } \\
		&=\int_U \left(e^{g^\infty(x,a')}-e^{g^n(x, a')}\right)da' \frac{ \pi^n(x,a) }{\int_U \exp[g^\infty(x,a')]da'}  +   \left(e^{g^n(x,a)}-e^{g^\infty(x, a)}\right) \frac{ 1 }{\int_U e^{g^\infty(x,a')}da'}.
	\end{aligned}
	$$
	Applying \eqref{eq:lm.C12est0} to the two terms respectively in the last line above, together with  \eqref{eq:lm.C12est0'} and \eqref{eq:lm.C12est1}, we get that
	\be\label{eq:lm.C12est3}  
	\sup_{a\in U} \|\pi^n(\cdot,a)-\pi^\infty(\cdot,a)\|_{\Cc_{\alpha}(B_{N_0+1}(0))}\leq A_0 \| D_x(w^n-w^\infty)(0,x)\|_{\Cc_{\alpha}(B_{N_0+1}(0))}.
	\ee 
	Then combining \eqref{eq:lm.C12est3}  with Assumptions \ref{assume.b.sigma}, \eqref{assume.r} and \eqref{eq:lm.C12est0} yields
	\be\label{eq:lm.C12est5}   
	\begin{aligned}
		& \|r^{\pi^n}-r^{\pi^\infty}\|_{\Cc_{\alpha/2,\alpha}(D_{N_0+1})}+ \|b^{\pi^n}-b^{\pi^\infty}\|_{\Cc_{\alpha}(B_{N_0+1}(0))}\\
		& \leq  A_0 \| D_x(w^n-w^\infty)(0,x)\|_{\Cc_{\alpha}(B_{N_0+1}(0))}\to 0.
	\end{aligned}
	\ee 
	Meanwhile, it holds that 
	$$
	\begin{aligned}
		\Hc(\pi^n(x))-\Hc(\pi^\infty(x))= & \int_U g^\infty(x,a) [\pi^\infty(x,a)-\pi^n(x,a)]da+\int_U [g^\infty(x,a)-g^n(x,a)] \pi^n(x,a)da \\
		&+\ln\left( \frac{\int_U \exp[g^n(x,a)] da} {\int_U \exp[g^\infty(x,a)]da} \right)
	\end{aligned}
	$$
	Applying \eqref{eq:lm.C12est0'} and \eqref{eq:lm.C12est0} to the first two integrals on the right-hand side, together with \eqref{eq:lm.C12est7'}, gives that
	$$
	\|\Hc(\pi^n(x))-\Hc(\pi^\infty(x))\|_{\Cc_\alpha(B_{N_0+1}(0))}\leq A_0\|D_x (w^n-w^\infty)\|_{\Cc_\alpha(B_{N_0+1}(0))},
	$$
	which, together with \eqref{eq:assume.delta}, yields 
	\bee\label{eq:lm.C12est7}   
	\|\lambda\delta (\Hc(\pi^n)-\Hc(\pi^\infty))\|_{\Cc_{\alpha/2, \alpha}(D_{N_0+1})}\leq A_0\|D_x (w^n-w^\infty)\|_{\Cc_\alpha(B_{N_0+1}(0))}\to 0.
	\eee 
	%Plugging \eqref{eq:lm.C12est5}  and \eqref{eq:lm.C12est7}  into \eqref{eq:lmC12.fn}, we readily arrive at \eqref{eq:lmC12.fnvanish}.  
\end{proof}

\begin{Lemma}\label{lm:Phi.continuous}
	Fix $0<\lambda<\lambda_0$, we have that, for any $0<\beta<\alpha$,
	\bi  
	\item[(i)] $\Mc_\lambda$ is a compact subset of the space $\Cc^{1,2, \wg}_{\beta/2, \beta}(\T\times \R^d)$; 
	\item[(ii)] $\Phi_\lambda$ is a continuous mapping from $\Mc_\lambda \to \Mc_\lambda$ equipped with norm $\|\cdot\|_{\Cc^{1,2, \wg}_{\beta/2,\beta}(\T\times \R^d)}$. 
	\ei 
\end{Lemma}

\begin{proof}
	{\bf Part (i).} Fix an arbitrary $0<\beta<\alpha$. Suppose that $(w^n)_{n\in \N}$ is a sequence in $\Mc_\lambda$, and we will carry out a diagonal argument. For each $N\in \N$, 
	\be\label{eq:lm.phicontinuous0}  
	\sup_{n\in \N}\|w^n\|_{\Cc^{0,1}_{\alpha/2, \alpha}(D_N)}\leq A^*, \sup_{n\in \N}\|w^n\|_{\Cc^{1,2}_{\alpha/2, \alpha}(D_N)}\leq A_\lambda.
	\ee 
	Take $N=1$, in view that $C^{l,k}_{\alpha/2,\alpha}(D_1)$ is compactly embedded in $C^{l,k}_{\beta/2,\beta}(D_1)$, we can find a subsequence of $(w^n)_{n\in \N}$, denoted by $(w^{1, n})_{n\in \N}$, and a function $u^1\in \Cc^{1,2}(D_1)$ such that $\|w^{1, n}-u^1\|_{\Cc^{1,2}_{\beta/2, \beta}(D_1)}\to 0$. In particular, all derivatives $\partial^l_t D^a_x w^{1,n}$ (here $a$ is a multi-index) with $l+|a|_{l^1}\leq 2$
	converge uniformly to the corresponding derivatives of $u^1$. Then, for any $0\leq 2l+|a|_{l^1}\leq 2$, it holds that
	$$
	\frac{|\partial^l D^a_x u^1(t,x)-\partial^l D^a_x u^1(s,y)|}{(|t-s|+|x-y|^2)^{\alpha/2}}\leq \limsup_{n\to\infty} \frac{|\partial^l D^a_x w^{1,n}(t,x)-\partial^l D^a_x w^{1,n}(s,y)|}{(|t-s|+|x-y|^2)^{\alpha/2}}.
	$$
	Therefore, $u^1\in \Cc^{1,2}_{\alpha/2,\alpha}(D_1)$ and 
	\be\label{eq:lm.phicontinuous1} 
	\|u^1\|_{\Cc^{0,1}_{\alpha/2, \alpha}(D_1)}\leq A^*,\quad \|u^1\|_{\Cc^{1,2}_{\alpha/2, \alpha}(D_1)}\leq A_\lambda.
	\ee 
	Repeating this argument for $N=i$, we can pick a subsequence of $(w^{i-1,n})_{n\in\N}$, denoted by $(w^{i, n})_{n\in \N}$, and a function $u^i\in \Cc^{1,2}(D_i)$  such that $\|w^{i, n}-u^i\|_{\Cc^{1,2}_{\beta/2, \beta}(D_i)}\to 0.$  Moreover, $\|u^i\|_{\Cc^{0,1}_{\alpha/2, \alpha}(D_i)}\leq A^*$ and $\|u^i\|_{\Cc^{1,2}_{\alpha/2, \alpha}(D_i)}\leq A_\lambda.$
	Then, we define the function $w$ on $\T\times \R^d$ by 
	$$
	w(t,x):= u^i(t,x) \text{   on }D_i,\quad \text{for }i=1,2,\cdots.
	$$ 
	As
	$
	u^{i+1}=u^i \text{ on } D_i,
	$
	$w$ is well-defined and satisfies that 
	$$
	\|w\|_{\Cc^{0,1}_{\alpha/2, \alpha}(D_N)}\leq  A^*, \quad \|w\|_{\Cc^{1,2}_{\alpha/2, \alpha}(D_N)}\leq A_\lambda \quad \forall N\in \N,
	$$
	which yields that $\|w\|_{\Cc^{0,1}_{\alpha/2, \alpha}(\T\times \R^d)}\leq A^*$ and $\|w\|_{\Cc^{1,2}_{\alpha/2, \alpha}(\T\times \R^d)}\leq A_\lambda$. It thus holds that $w\in \Mc_\lambda$. Moreover, the diagonal subsequence $(w^{n,n})_{n\in \N}$ satisfies that 
	$$
	\lim_{n\to \infty}\|w^{n,n}-w\|_{\Cc^{1,2}_{\beta/2, \beta}(D_N)}=0,\quad \forall N\in \N.
	$$
	Then, for any $\eps>0$, there exist $N_0$ and $n_0$ such that
	$$\sum_{N=N_0+1}^\infty \frac{ A_\lambda}{2^N}\leq \frac{\eps}{4}, \;\text{ and }  \sum_{N=1}^{N_0}\frac{1}{2^N}\|w^{n,n}-w\|_{\Cc^{1,2}_{\beta/2, \beta}(D_{N_0})}\leq \frac{\eps}{2}\; \text{ $\forall n\geq n_0$},$$
	which imply that, for all $n\geq n_0$,
	\begin{align*}
		&\|w^{n,n}-w\|_{\Cc^{1,2, \wg}_{\beta/2, \beta}(\T\times \R^d)}\\
		&\leq \sum_{N=1}^{N_0}\frac{1}{2^N} \|w^{n,n}-w\|_{\Cc^{1,2}_{\beta/2, \beta}(D_{N_0})} +\sum_{N=N_0+1}^\infty \left(\|w^{n,n} \|_{\Cc^{1,2}_{\beta/2, \beta}(D_{N})} +\|w\|_{\Cc^{1,2}_{\beta/2, \beta}(D_{N})} \right) \\
		&\leq  \sum_{N=1}^{N_0}\frac{1}{2^N} \|w^{n,n}-w\|_{\Cc^{1,2}_{\beta/2, \beta}(D_{N_0})} +\sum_{N=N_0+1}^\infty \left(\|w^{n,n} \|_{\Cc^{1,2}_{\alpha/2, \alpha}(D_{N})} +\|w\|_{\Cc^{1,2}_{\alpha/2, \alpha}(D_{N})} \right) \\
		&\leq  \frac{\eps}{2}+\frac{\eps}{4}+\frac{\eps}{4}=\eps.
	\end{align*}
	In conclusion, $w$ is the limit of $(w^{n,n})_{n\in \N}$ within $\Mc_\lambda$, and hence $\Mc_\lambda$ is a  compact subset of the space $\Cc^{1,2, \wg}_{\beta/2, \beta}(\T\times \R^d)$.
	
	{\bf Part (ii).} Take an arbitrary function $w\in \Mc_\lambda$, we have that  $\|w(0,\cdot)\|_{\Cc^1_\alpha(\R^d)}\leq \|w\|_{\Cc^{0,1}_{\alpha/2,\alpha}(\T\times \R^d)}\leq A^*$, then Lemma \ref{prop:iteration.estimateentropy} shows that $\Phi_\lambda(w)\in \Mc_\lambda$. Consequently, $\Phi_\lambda$ maps from $\Mc_\lambda$ to $\Mc_\lambda$. We next show that $\Phi_\lambda$ is a continuous mapping from $\Mc_\lambda$ to $\Mc_\lambda$, equipped with norm $\|\cdot\|_{\Cc^{1,2, \wg}_{\alpha/2,\alpha}(\T\times \R^d)}$. Then the desired result holds due to $\beta<\alpha$. 
	
	Take a sequence $(w^n)_{n\in \N\cup\{\infty\}}\subset \Mc_\lambda$ with $\lim_{n\to\infty}\|w^n-w^\infty\|_{\Cc^{1,2,\wg}_{\alpha/2,\alpha}(\T\times \R^d)}\to 0$. Set $V^n:=\Phi_\lambda(w^n), \pi^n(x,a)=\Gamma_\lambda(x,D_xw^n(0,x),a)$ for each $n\in \N\cup\{\infty\}$. By Lemma \ref{lm:entropy.ln}, it holds that
	$$
	\partial_t V^n(t,x)+  \frac{1}{2}\tr\left((\sigma\sigma^T)(x) D^2_x V^n(t,x)\right)+ b^{\pi^n}(x)D_x V^n(t,x) +r^{\pi^n}(t,x)+\lambda \delta(t)\Hc(\pi^n(x))=0.
	$$
	Define $\bar V^n=V^\infty-V^n$ for $n\in \N$, then
	\be\label{eq:PDF.Vbar}  
	\begin{aligned}
		0=\partial_t \bar V^n(t,x)+b^{\pi^\infty}(x)D_x \bar V^n(t,x)+  \frac{1}{2}\tr\left((\sigma\sigma^T)(x) D^2_x \bar V^n(t,x)\right)+ f^n(t,x),
	\end{aligned}
	\ee 
	with
	\be\label{eq:lmC12.fn}
	f^n(t,x):=  r^{\pi^n}(t,x)-r^{\pi^\infty}(t,x)+\lambda \delta(t)[\Hc(\pi^n(x))-\Hc(\pi^\infty(x))]+[b^{\pi^\infty}(x)-b^{\pi^n}(x)]D_x V^n(t, x).
	\ee
	We now claim that $\|\bar V^n\|_{\Cc^{1,2,\wg}_{\alpha/2,\alpha}(\T\times \R^d)} \to 0$. 
	Given any $\eps>0$, there exists some $N_0$ such that $\sum_{N=N_0+1}^\infty\frac{A_\lambda}{2^N} \leq \frac{\eps}{4}$, then $(V^n)_{n\in \N\cup\{\infty\}}\in \Mc_\lambda$ gives that 
	$$
	\sum_{N=N_0+1}^\infty \frac{1}{2^N}\|\bar V^n\|_{\Cc^{1,2}_{\alpha/2,\alpha}(D_N)}\leq \sum_{N=N_0+1}^\infty \frac{1}{2^N} \left( \|V^n\|_{\Cc^{1,2}_{\alpha/2,\alpha}(D_N)}+  \|V^\infty\|_{\Cc^{1,2}_{\alpha/2,\alpha}(D_N)} \right)\leq 2\sum_{N=N_0+1}^\infty\frac{A_\lambda}{2^N}\leq \frac{\eps}{2}.
	$$
	To prove the claim, it suffices to show 
	\be\label{eq:lm.continuous1} 
	\lim_{n\to \infty} \|\bar V^n\|_{\Cc^{1,2}_{\alpha/2,\alpha}(D_{N_0})} =0.
	\ee
	We begin by showing that 
	\be\label{eq:lmC12.vbarbound}  
	\lim_{n\to \infty}\|\bar V^n\|_{L^\infty (D_{N_0+1})}=0.
	\ee 
	Define $\rho^n_N:=\inf\{s\geq 0: X^{\pi^n}_s\notin B_{N}(0)\}$ for $n\in \N\cup\{\infty\}$ and
	note that $\sup_{n\in \N\cup\{\infty\}}\|b^{\pi^n}\|_{\Cc_\alpha(\R^d)}$ is bounded. This, together with the condition on $\sigma$ in Assumption \ref{assume.b.sigma} and the uniform bound on $\|V^n\|_{L^\infty(\T\times \R^d)}$, ensures that, for any $\eps>0$, there exists a sufficiently large  $N_1\in \N$ with $N_1\geq N_0+1$ such that 
	\be\label{eq:lmvbarN.eps} 
	\|V^n(t,x)-V^{n}_{N_1}(t,x)\|_{L^\infty(D_{N_0+1})}\leq \eps\quad \forall n\in \N\cup\{\infty\}.
	\ee  
	Here, the truncated value functions $V^n_{N_1}$ for $n\in \N\cup\{\infty\}$ are defined by 
	$$
	V^n_{N_1}(t,x):=E_x\left[\int_0^{\rho^{n}_{N_1}\wedge (N_1-t)}  \left( r^{\pi^n}(t+s, X^{\pi^n}_s)+\delta(t+s)\lambda \Hc(\pi^n(X^{\pi^n}_s))\right)ds \right], \text{ for } (t,x)\in D_{N_0+1}.
	$$
	It then holds that
	$$
	\begin{aligned}
		&|   V^n_{N_1} (t,x)-  V^\infty_{N_1} (t,x) | \\
		& =E_x\left[\int_0^{\rho^{n}_{N_1}\wedge (N_1-t)}  \left( [r^{\pi^n}-r^{\pi^\infty}](t+s, X^{\pi^n}_s)+\delta(t+s)\lambda [\Hc(\pi^n)-\Hc(\pi^\infty)](X^{\pi^n}_s) \right)ds \right]  \quad \text{(I)}\\
		&\quad + E_x\bigg[\int_0^{\rho^{n}_{N_1}\wedge (N_1-t)}  \left( r^{\pi^\infty}(t+s, X^{\pi^n}_s)+\delta(t+s)\lambda \Hc(\pi^\infty(X^{\pi^n}_s))\right)ds\\
		&\qquad \qquad  - \int_0^{\rho^{\infty}_{N_1}\wedge (N_1-t)}  \left( r^{\pi^\infty}(t+s, X^{\pi^\infty}_s)+\delta(t+s)\lambda \Hc(\pi^\infty(X^{\pi^\infty}_s))\right)ds  \bigg]   \qquad \qquad \qquad\text{(II)}
	\end{aligned}
	$$
	By Lemma \ref{lm:Phi.prepare}, $\|r^{\pi^n}-r^{\pi^\infty}\|_{\Cc_{\alpha/2, \alpha}(D_{N_1})}$, $\|\delta\Hc({\pi^n})-\delta\Hc({\pi^\infty})\|_{\Cc_{\alpha/2, \alpha}(D_{N_1})}$ and $\|b^{\pi^n}-b^{\pi^\infty}\|_{\Cc_{ \alpha}(B_{N_1}(0))}$ all tend to zero as $n\to\infty$. As a result, the term (I) tends to zero as $n\to\infty$ uniformly for $(t,x)\in D_{N_0+1}$. Also, by taking $X^{\pi^n}_0=x$ for all $n\in \N\cup\{\infty\}$, the standard stability theory of SDEs yields the following convergences hold uniformly over the starting position $x\in B_{N_0+1}(0)$:
	$$
	\sup_{0\leq s\leq N_1} |X^{\pi^n}-X^{\pi^\infty}|\to 0 \text{ in probability }, \rho^n_{N_1}\to \rho^\infty_{N_1} \text{ in probability}.
	$$
	Consequently, the term (II) also tends to zero uniformly for $(t,x)\in D_{N_0+1}$. In sum, we conclude that 
	$$
	\sup_{(t,x)\in D_{N_0+1}} |V^n_{N_1} (t,x)-  V^\infty_{N_1} (t,x)| \to 0 \text{ as $n\to\infty$}.
	$$
	This, together with \eqref{eq:lmvbarN.eps}, gives the desired \eqref{eq:lmC12.vbarbound}.
	
	By Lemma \ref{lm:Phi.prepare} and that $\sup_{n\in \N}\|V^n\|_{\Cc^{0,1}_{\alpha/2,\alpha/2}(\T\times \R^d)}\leq A^*$, $f^n$ defined in \eqref{eq:lmC12.fn} satisfies that
	\be\label{eq:lmC12.fnvanish}  
	\|f^n\|_{\Cc_{\alpha/2, \alpha}(D_{N_0+1})}\to 0 \text{ as $n\to \infty$}.
	\ee 
	Then, applying Schaulder estimate on \eqref{eq:PDF.Vbar} (see, e.g., \cite[Theorem 8.11.1 and Remark 8.11.2]{krylov1996lectures}) readily yields 
	\be\label{eq:lm.barvest} 
	\|\bar V^n\|_{\Cc^{1,2}_{\alpha/2,\alpha}(D_{N_0})}\leq C\left(  \|\bar V^n\|_{L^\infty(D_{N_0+1})} + \|f^n\|_{\Cc_{\alpha/2 ,\alpha}(D_{N_0+1})} \right),
	\ee
	where $C$ is a finite constant depending on $A_\lambda$ and $\|b^{\pi^\infty}\|_{\Cc_{\alpha}(B_{N_0+1}(0))}$.  \eqref{eq:lmC12.vbarbound} and \eqref{eq:lmC12.fnvanish} together imply that the right-hand side of \eqref{eq:lm.barvest} tends to zero as $n\to\infty$, which completes the proof.
\end{proof}

We are now ready to prove Theorem \ref{thm:equil.existenceentropy} Part (iii). 

\begin{proof}[{\bf Proof of Theorem \ref{thm:equil.existenceentropy} Part (iii):}]
	Fix an entropy weight $0<\lambda\leq \lambda_0$ and an arbitrary H\"older constant $0<\beta<\alpha$.  Lemmas \ref{prop:iteration.estimateentropy} and \ref{lm:Phi.continuous} show that the mapping $\Phi_\lambda$ defined in \eqref{eq:Phi.fixedpoint} is a continuous mapping from the compact convex set $\Mc_\lambda$, equipped with norm $\|\cdot\|_{\Cc^{1,2, \wg}_{\beta/2,\beta}(\T\times \R^d)}$, to itself. Hence, by Schauder fixed-point theorem (see, e.g., \cite[Theorem 5.28]{Rudin1991functional}), $\Phi_\lambda$ has a fixed point $w\in \Mc_\lambda$ such that $w=V^{\pi^*}_\lambda$ with $\pi^*(x,a)= \Gamma_{\lambda}(x, D_xw(0,x), a)$. As a result, Lemma \ref{prop:iteration.estimateentropy} yields that $V^{\pi^*}_\lambda$ satisfies \eqref{eq:cha.HJBentropy'}. By taking $t=0$ in \eqref{eq:prop.iteraPDE},  together with the definition of $\Gamma_\lambda$ in \eqref{eq:def.Gammalambda}, we see that $V^{\pi^*}_\lambda$ also satisfies \eqref{eq:cha.HJBentropy}. Then, we can conclude from Theorem \ref{thm:equil.existenceentropy} Part (ii) that $\pi^*$ is a regularized equilibrium with the entropy parameter $\lambda$. 
	Moreover, the estimates of $V^{\pi^*}_\lambda$ are direct consequences from $V^{\pi^*}_\lambda\in \Mc_\lambda$ and Lemma \ref{prop:iteration.estimateentropy}.
\end{proof}

\begin{Remark}
	We emphasis that, the existence of a classical solution to the EEHJB system \eqref{eq:cha.HJBentropy}--\eqref{eq:cha.HJBentropy'} (i.e., the existence of an equilibrium under entropy regularization) stated in Theorem \ref{thm:equil.existenceentropy} Part (iii) is not restricted to the case of $\lambda\in(0,\lambda_0]$, but holds valid for all $\lambda>0$ by simply allowing $A^*$ to depend on $\lambda$. First, the sub-linear growth estimate in Lemma \ref{lm:entropy.ln} holds for any $\lambda>0$. Indeed, as stated in \cite[Lemma 1]{bayraktar2025relaxed}, for an arbitrary $\lambda>0$,
	$$
	|\Hc(\Gamma_\lambda(x,p,a))|=\left| \int_U \ln(\Gamma_{\lambda}(x,p,a))\Gamma_{\lambda}(x,p,a)da\right|\leq \kappa+\ell \ln(1+|p|),\quad \forall y\in \R^d,
	$$
	where $\kappa>0$ depends on $\lambda$, $\ell$, $\Leb(U)$ and the parameters in Assumption \ref{assume.lipsa.U}. When $\lambda\leq \lambda_0$, $\kappa$ can be further refined as $\kappa= K_6+K_7 |\ln\lambda|$ with $K_6, K_7$ independent of $\lambda$. Thus, for any $\lambda>0$, the fixed point argument in this section remains valid.
	
	In the analysis above, we have deliberately chosen the estimates of the value function $V^{\pi}$ to be independent of $\lambda$ because these estimates serve as a key tool in the convergence analysis of the next section as $\lambda\rightarrow 0$.
	
\end{Remark}

\section{Existence of Equilibrium by Vanishing Entropy Regularization}\label{sec:convergence}

In this section, we return to the original time-inconsistent control problem without entropy regularization. To address the existence of equilibrium, we aim to establish a core convergence result: as the entropy parameter  $\lambda\rightarrow 0$, solutions of the  EEHJB system converge to a solution to the original EHJB system in a proper sense such that the limit of the regularized equilibrium is a relaxed equilibrium in the original problem. 

Throughout this section, we denote $q=\frac{d+2}{d+1+\alpha}=\frac{p}{1-p}$ (such that $1/p+1/q=1$), and impose Assumptions \ref{assume.r}, \ref{assume.b.sigma}, \ref{assume.lipsa.U} and the conditions on $\delta$ in \eqref{eq:assume.delta}--\eqref{eq:assume.deltaholder}.

The next main result gives a sufficient condition, new to the literature, for the existence of an equilibrium based on the existence of a strong (rather than classical) solution to the EHJB system.
\begin{Theorem}[A sufficient condition for equilibrium]\label{thm:equi.existence1} For a measurable strategy $\pi^*: \R^d\to \Pc(U)$, if $V^{\pi^*}\in \Cc^{0,1 }_{\alpha/2, \alpha}(\T\times \R^d)\cap  W^{1,2, \ul}_{p}(\T\times \R^d)$ with $\sup_{x\in\R^d}\|V^{\pi^*}(\cdot,x)\|_{\Cc^{1}_{\alpha/2}(\T\times \R^d)}<\infty$, and $V^{\pi^*}$ is a strong solution to the EHJB system 
	\begin{align}
		0=&\partial_t u(0,x)+\frac{1}{2}\tr\left(\sigma\sigma^T(x) D^2_x u(0,x)\right) \notag\\ 
		&+\sup_{\varpi\in \Pc(U)} \left\{ \int_U \left[ b(x,a)D_x u(0,x) +r(0,x,a) \right] \varpi(da) \right\}  \; \text{a.e on $\R^d$},\label{eq:cha.weakHJB}\\
		0=&\partial_t u (t,x)+\frac{1}{2}\tr\left((\sigma\sigma^T)(x) D^2_x u (t,x)\right) +b^{\pi^*}(x)D_x u(t,x) +r^{\pi^*}(t,x) \label{eq:cha.weakHJB'} \; \text{a.e. on $\T\times\R^d$},
	\end{align}
	then $\pi^*$ is an equilibrium in Definition \ref{def:equi.relaxed}. 
\end{Theorem}

Next, under our previous model assumptions, we further show that the existence of a strong solution to the EHJB system \eqref{eq:cha.weakHJB}--\eqref{eq:cha.weakHJB'} is always guaranteed, thereby ensuring the existence of equilibrium in the original time-inconsistent control problem even when the classical solution of the EHJB system is unavailable.  

\begin{Theorem}[Existence of a strong solution to EHJB system]\label{thm:equi.existence}
	There exists a measurable strategy $\pi^*: \R^d\to \Pc(U)$ such that $V^{\pi^*}\in \Cc^{0,1 }_{\alpha/2, \alpha}(\T\times \R^d)\cap  W^{1,2, \ul}_{p}(\T\times \R^d)$ with $\sup_{x\in\R^d}\|V^{\pi^*}(\cdot,x)\|_{\Cc^{1}_{\alpha/2}(\T\times \R^d)}<\infty$, and $V^{\pi^*}$ is a strong solution to the EHJB system \eqref{eq:cha.weakHJB}--\eqref{eq:cha.weakHJB'}, and hence $\pi^*$ is an equilibrium in Definition \ref{def:equi.relaxed} by Theorem
	\ref{thm:equi.existence1}.  
\end{Theorem}

In the following subsections, we first prove Theorem \ref{thm:equi.existence} by showing that the classical solutions to the EEHJB system will converge to a strong solution to the EHJB system when the entropy regularization vanishes. Then, by leveraging the convergence results, we will further develop some new verification arguments to etablish the weaker sufficient condition for the existence of equilibrium in Theorem \ref{thm:equi.existence1}. 

%\begin{Remark}
%We want to mention that the vanishing entropy framework for showing the existence of equilibrium in the current paper can also be applied in a much more general setting, including finite horizon, dependence on initial data, nonlinear structure (mean-variance terminal reward)....... And we will study in future.
%\end{Remark}

%``Take a feedback strategy $\pi^*: \R^d\mapsto \Pc(U)$. Then $V^{\pi^*}\in \Cc^{0,1 }_{\alpha/2, \alpha}(\T\times \R^d)\cap  W^{1,2, \ul}_{p}(\T\times \R^d)$ with $\sup_{x\in\R^d}\|V^{\pi^*}(\cdot,x)\|_{\Cc^{1}_{\alpha/2}(\T\times \R^d)}<\infty$, and $V^{\pi^*}$ is a strong solution to \eqref{eq:cha.weakHJB'}. Moreover, the following two statements are equivalent.
%	\begin{itemize}
	%		\renewcommand{\labelitemii}{$\bullet$}
	%	\item $\pi^*$ is an equilibrium satisfying \ref{eq:def.equi} in Definition \ref{def:equi.relaxed};
	%		\item $V^{\pi^*}$ satisfies \eqref{eq:cha.weakHJB}. 
	%	\end{itemize}
%	As a result, an equilibrium exists satisfying \ref{eq:def.equi} in Definition \ref{def:equi.relaxed}."

\subsection{Proof of Theorems \ref{thm:equi.existence}}

Let us take a sequence $(\lambda_n)_{n\in \N}$ with $\lambda_n\to 0+$. For each $n\in \N$, Theorem \ref{thm:equil.existenceentropy} Part (iii) states the existence of a regularized equilibrium $\pi^n$ with entropy parameter $\lambda_n$ in Gibbs form such that $\|V^{\pi^n}_{\lambda_n}\|_{\Cc^{0,1 }_{\alpha/2,\alpha}(\T\times \R^d)}\leq A^*$ for all $n\in \N$. Let us further denote $v^n(t,x):=V^{\pi^n}_{\lambda_n}(t,x)$ for each $n$. It then follows from Theorem \ref{thm:equil.existenceentropy} Part (i) that $(v^n, \pi^n)$ satisfies
\begin{align}
	0=& \partial_t v^n(0,x)+\frac{1}{2}\tr\left(\sigma\sigma^T(x,a) D^2_x v^n(0,x)\right)  
	+\lambda_n \ln \left\{ \int_U \exp\left(\frac{1}{\lambda}\left[ b(x,a)D_x v^n(0,x) +r(0,x,a)\right] \right) \right\}, \label{eq:HJBn}\\
	0=&\partial_t v^n(t,x)+\frac{1}{2}\tr\left((\sigma\sigma^T)^{\pi^n}(x) D^2_x v^n(t,x)\right) 
	+  b^{\pi^n}(x)D_x v^n(t,x) +r^{\pi^n}(t,x)+\lambda_n\delta(t)\Hc(\pi^n(x)).   \label{eq:HJBn'}
\end{align}
And we will construct a strong (instead of classic) solution to the EHJB system \eqref{eq:cha.weakHJB}--\eqref{eq:cha.weakHJB'} by searching a subsequence limit of $(v^n, \pi^n)_{n\in \infty}$.

\begin{Lemma}\label{lm:thm.C12andyoung}
	There exist a subsequence $(v^{n_k}, \pi^{n_k})_{k\in \N}$ and a pair $(v^\infty, \pi^\infty)$ satisfying the following.
	\begin{itemize}
		\item[(i)] $v^\infty\in \Cc^{0,1 }_{\alpha/2, \alpha}(\T\times \R^d)\cap  W^{1,2, \ul}_{p}(\T\times \R^d)$, and for  each $N\in \N$ and any test function $\phi\in L^q(D_N)$, it holds that
		\bee\label{eq:lm.converge}  
		\begin{aligned}
			&\lim_{k\to\infty}\|v^{n_k}-v^\infty\|_{\Cc^{0,1 }_{\beta/2, \beta}(D_N)}=0 \quad\forall\,0<\beta<\alpha,\\ 
			&\lim_{k\to\infty}\int_{D_N} \phi(t,x) \left(\partial^l_tD^a_{x} v^{n_k}(t,x)-\partial^l_tD_{x}^a v^\infty(t,x) \right) dtdx = 0\quad \forall 0\leq 2l+|a|_{l^1}\leq 2.
		\end{aligned}
		\eee 
		\item[(ii)] $\pi^\infty: \R^d\to \Pc(U)$ is Borel measurable, and $\pi^{n_k}$ converges to $\pi^\infty$ in the sense that
		\be\label{eq:convergence.young}  
		\lim_{k\to\infty}\int_{\R^d} \left(\int_U \phi(x,a)\pi^{n_k}(x,a)da\right)dx = \int_{\R^d} \left(\int_U \phi(x,a)\pi^\infty(x,da)\right)dx
		\ee 
		for any test function $\phi(x,a):\R^d\times U\to \R^d$ that is continuous on the control variable $a$ for each $x$ and satisfies
		$
		\int_{\R^d} \max_{a} |\phi(x,a)| dx<\infty.
		$
		Moreover,
		\be\label{eq:lm.brconverge}  
		b^{\pi^{n_k}}(\cdot) \to b^{\pi^{\infty}}(\cdot) \text{ and }  r^{\pi^{n_k}}(t,\cdot)\to  r^{\pi^{\infty}} (t,\cdot)  \text{ in the weak-$*$ topology of $L^\infty(\R^d)$}.
		%\int_\R^d b^{\pi^{n_k}}(x) \tilde\phi(x)dx\to \int_\R^d b^{\pi^{\infty}}(x) \tilde\phi(x)dx, \quad \int_\R^d r^{\pi^{n_k}}(t,x) \tilde\phi(x)dx\to \int_\R^d r^{\pi^{\infty}} (t,x)\tilde\phi(x)dx,
		\ee 
		\item[(iii)] $V^{\pi^\infty}=v^\infty$, $\sup_{x\in \R^d}\|v^{\infty}(\cdot, x)\|_{\Cc^1_{\alpha/2}(\T)}\leq A^*$. %and it holds that
		%\be\label{eq:vnvinftyt.equi}  
		%\lim_{k\to\infty}\sup_{x\in \R^d}\|v^{n_k}(\cdot, x)-v^{\infty}(\cdot, x)\|_{\Cc^1_{\alpha/2}(\T)}=0.
		%\ee 
	\end{itemize}
\end{Lemma}

\begin{proof}
	{\bf Part (i).} We first show that the limit of a subsequence exists in the space $\Cc^{0,1 }_{\alpha/2, \alpha}(\T \times \R^d)\cap W^{1,2, \ul}_{p}(\T\times \R^d)$ using a diagonal argument. 
	Thanks to Theorem \ref{thm:equil.existenceentropy}, for each $N$, it holds that
	$$
	\sup_{n\in \N}\|v^{n}\|_{\Cc^{0,1 }_{\alpha/2, \alpha}(D_N)}\leq A^*,  \sup_{n\in \N}\|v^{n}\|_{ W^{1,2}_{p}(D_N)}\leq A_N A^*,
	$$
	where $A^*$ is the finite constant stated in Theorem \ref{thm:equil.existenceentropy}, and $A_N=C (N^{d+1})^{1/p}$ with $C$ being a finite constant independent of $d, p$ and $N$. For each $N$, define 
	$$\Ec_N:= \left\{w\in \Cc^{0,1 }_{\alpha/2, \alpha}(D_N)\cap W^{1,2}_{p}(D_N): \|w\|_{\Cc^{0,1 }_{\alpha/2, \alpha}(D_N)}\leq A^*,  \|w\|_{ W^{1,2}_{p}(D_N)}\leq A_N A^* \right\}.$$
	Take $(v^{0,n})_{n\in \N}=(v^n)_{n\in \N}$ and fix an arbitrary $0<\beta<\alpha$. We shall repeat the following discussion for each $N=0,1,2,\cdots$. Because $\Cc^{0,1 }_{\alpha/2, \alpha}(D_N)$ is compactly embedded in  $\Cc^{0,1 }_{\beta/2, \beta}(D_N)$, by similar arguments for \eqref{eq:lm.phicontinuous0} and \eqref{eq:lm.phicontinuous1}, we can find a subsequence $(v^{N-1, n_k})_{k\in \N}$ of $(v^{N-1,n})_{n\in N}$ and $u^N\in \Ec_N$ such that $\|u^N\|_{\Cc^{0,1}_{\alpha/2,\alpha}(D_N)}\leq A^*$ and
	$$
	\lim_{k\to\infty}\|v^{N-1, n_k}-u^N\|_{\Cc^{0,1 }_{\beta/2, \beta}(D_N)}= 0.
	$$
	{Since $\sup_{n\in \N}\|v^{N-1, n_k}\|_{W^{1,2}_{p}(D_N)}\leq A_NA^*$, by Banach–Alaoglu theorem, we can extract a subsequence of $(v^{N-1, n_k})_{k\in \N}$, denoted by $(v^{N, n})_{n\in \N}$, such that, for any test function $\phi\in L_q(D_N)$, 
		$$
		\begin{aligned}
			&\lim_{n\to\infty}\int_{D_N} \phi(t,x)\partial_t v^{N,n}(t,x)dt dx =\int_{D_N} \phi(t,x) w(t,x)dt dx,\\
			&\lim_{n\to\infty}\int_{D_N} \phi(t,x)\partial^{2}_{x_ix_j} v^{N,n}(t,x)dt dx =\int_{D_N} \phi(t,x) w^{ij}(t,x)dt dx, \quad \forall 1\leq i,j\leq d,
		\end{aligned}
		$$
		where the functions $w, w^{ij}$ belong to $L_p(D_N)$. Then for any test function $\phi\in C^\infty_c(D_N)$, $\partial_t \phi, \partial_{x_i} \phi\in L_q(D_N)$ and for each $1\leq i,j\leq d$,
		$$
		\begin{aligned}
			-\int_{D_N} \phi(t,x) w^{ij}(t,x)dt dx=&- \lim_{n\to\infty}\int_{D_N} \phi(t,x)\partial^{2}_{x_ix_j} v^{N,n}(t,x)dt dx\\ =&\lim_{n\to\infty}\int_{D_N} \partial_{x_i}\phi(t,x)\partial_{x_j} v^{N,n}(t,x)dt dx = \int_{D_N} \partial_{x_i}\phi(t,x)\partial_{x_j} u^N(t,x)dt dx.
		\end{aligned}
		$$
		Thus, $w^{ij}$ is the weak derivative $\partial_{x_i} (\partial_{x_j} u^N)$ for each $1\leq i,j\leq d$. A similar argument shows that $w$  is the weak derivative $\partial_t u^N$. Therefore, $u^N\in W^{1,2}_p(D_N)$, and 
		\be\label{eq:lm.C12} 
		\lim_{n\to\infty}\int_{D_N} \phi(t,x) \left(\partial^l_tD^a_{x} v^{N, n}(t,x)-\partial^l_tD_{x}^a u^N(t,x) \right) dtdx = 0\quad \forall 0\leq 2l+|a|_{l^1}\leq 2, \forall \phi\in L^q(D_N).
		\ee}
	%{\color{red}Due to $\sup_{n\in \N}\|v^{N-1, n_k}\|_{W^{1,2}_{p}(D_N)}\leq A_NA^*$, by Banach–Alaoglu theorem, we can conclude  that $u^N\in W^{1,2}_{p}(D_N)$ with $\|u^N\|_{W^{1,2}_{p}(D_N)}\leq A_NA^*$ and there exists a subsequence of $(v^{N-1, n_k})_{k\in \N}$, denoted by $v^{N, n}_{n\in \N}$, such that 
		%$\|\tilde u^N\|_{W^{1,2}_p(D_N)}\leq A_N A^*$ and 
		%	\be\label{eq:lm.C12} 
		%	\lim_{n\to\infty}\int_{D_N} \phi(t,x) \left(\partial^l_tD^a_{x} v^{N, n}(t,x)-\partial^l_tD_{x}^a u^N(t,x) \right) dtdx = 0\quad \forall 0\leq 2l+|a|_{l^1}\leq 2, \forall \phi\in L^q(D_N).
		%	\ee}
	As a result, \eqref{eq:lm.C12} and the following holds for the subsequence $(v^{N, n})_{n\in \N}$ and $u^N$
	$$
	\lim_{n\to\infty}\|v^{N, n}-u^N\|_{\Cc^{0,1 }_{\beta/2, \beta}(D_N)}= 0.
	$$
	Then, for each $N$, we obtain a limit $u^N$. Note that $u^{N+1}=u^N$ on $D_N$ for each $N\in \N$. Hence, a function $v^\infty\in \Cc^{0,1}_{\alpha/2,\alpha}(\T\times \R^d)$ exists such that $v^\infty=u^N\in W^{1,2}_p(D_N)$ on $D_N$ for each $N\in \N$ and
	$$
	\|v^\infty\|_{\Cc^{0,1 }_{\alpha/2, \alpha}(\T\times \R)}\leq A^*.
	$$
	By taking the diagonal subsequence $(v^{N,N})_{N\in \N}$, we obtain for each $K\in \N$, any $l\in \N$ and any multi-index $a$ satisfying $ 0\leq 2l+|a|_{l^1}\leq 2$,
	$$
	\begin{aligned}
		&\lim_{N\to\infty}\|v^{N,N}-v^\infty\|_{\Cc^{0,1 }_{\beta/2, \beta}(D_K)}=0,\quad\forall\,0<\beta<\alpha,\\ 
		&\lim_{N\to\infty}\int_{D_K} \phi(t,x) \left(\partial^l_tD^a_{x} v^{N,N}(t,x)-\partial^l_tD_{x}^a v^\infty(t,x) \right) dtdx = 0,\quad \forall \phi\in L^q(D_K).
	\end{aligned}
	$$
	Moreover, by fixing $(t_0,x_0)\in \T\times \R^d$, the second line above implies that, for each $l\in \N$ and the multi-index $a$ such that $0\leq 2l+|a|_{l^1}\leq 2$,
	$$
	\lim_{N\to\infty}\int_{D_1(t_0,x_0)} \phi(t,x) \left(\partial^l_tD^a_{x} v^{N, N}(t,x)-\partial^l_tD_{x}^a v^\infty(t,x) \right) dtdx = 0, \forall \phi\in L^q(D_1(t_0,x_0)).
	$$
	{Since for a generic function $g\in L^p(D_1(t_0,x_0))$,
		$$
		\|g\|_{L^p(D_1(t_0,x_0))}=\sup_{\|\phi\|_{L^q(D_1(t_0,x_0))}\leq 1} \int_{D_1(t_0,x_0)} \phi(t,x) g(t,x) dtdx,
		$$
		for any test function $\phi$ with $\|\phi\|_{L^q(D_1(t_0,x_0))}\leq 1$, we have
		$$
		\begin{aligned}
			\int_{D_1(t_0,x_0)} \phi(t,x) \partial^l_tD_{x}^a v^\infty(t,x) dtdx =&  \lim_{N\to\infty}\int_{D_1(t_0,x_0)} \phi(t,x) \partial^l_tD^a_{x} v^{N, N}(t,x) dtdx\\
			\leq & \liminf_{N\to\infty} \|\partial_t D^a_x v^{N,N}\|_{L^p(D_1(t_0,x_0))}.
		\end{aligned}
		$$
		It follows that 
		$$
		\|\partial_t D^a_x v^\infty\|_{L^p(D_1(t_0,x_0))} \leq  \liminf_{N\to\infty} \|\partial_t D^a_x v^{N,N}\|_{L^p(D_1(t_0,x_0))}\leq A^*\quad \forall 0\leq 2l+|a|_{l^1}\leq 2,
		$$
		%by $ \|v^{N,N}\|_{W^{1,2, \ul}_{p}(\T\times)}\leq A^*$, we have
		and hence, $\|v^\infty\|_{ W^{1,2,\ul}_{p}(\T\times\R^d)}\leq A^*$.}
	
	{\bf Part (ii).} Next, we consider the associated sequence $(\pi^{N,N})_{N\in \N}$ with respect to the sequence $(v^{N,N})_{N\in \N}$. By Young measure theory (see, e.g., \cite[Chapter IV]{warga2014optimal}), a subsequence $(\pi^{n_k})_{k\in \N}$ of $(\pi^{N,N})_{N\in \N}$ and a Borel function $\pi^\infty: \R^d\to \Pc(U)$ exist such that the convergences stated in part (ii) holds. 
	In conclusion, the corresponding subsequence $(v^{n_k}, \pi^{n_k})_{k\in \N}$ and the pair $(v^\infty, \pi^\infty)$ indeed satisfy \eqref{eq:lm.brconverge} and Part (i).
	
	% As a result, of the continuity of $b, r$ on $a$, it follows from \eqref{eq:convergence.young} that
	Moreover, for any test function $\phi\in L^1(\R^d)$, by Assumption \ref{assume.lipsa.U},  $a\mapsto b(x,a)\phi(x)$ is Lipschitz continuous for all $x$ and $\int_{\R^d} \max_{a\in U}|b(x,a)\phi(x)|dx\leq \|b(x,a)\|_{L^\infty(\R^d\times U)}\int_{\R^d}\phi(x)dx<\infty$.
	As a result, 
	$$
	\begin{aligned}
		\lim_{k\to\infty}\int_{\R^d} b^{\pi^{n_k}}(x)\phi(x)dx =&\lim_{k\to\infty}\int_{\R^d} \left( \int_U [b(x,a)\phi(x)]\pi^{n_k}(x,a)da\right)dx \\
		=& \int_{\R^d} \left( \int_U [b(x,a)\phi(x)]\pi^{\infty}(x,da)\right)dx 
		= \int_{\R^d} b^{\pi^{\infty}}(x)\phi(x)dx\quad \forall \phi\in L^1(\R^d).
	\end{aligned}
	$$
	A similar argument also holds for $r(t,\cdot)$. As a result, \eqref{eq:lm.brconverge}  holds.
	%for any $\tilde\phi\in L^1(\R^d)$.
	
	{\bf Part (iii). }  Now we show $V^{\pi^\infty}(t,x)=v^\infty(t,x)$ for any $(t,x)\in \T\times \R^d$. By Theorem \ref{thm:equil.existenceentropy} Part (iii), 
	$$
	\sup_{k\in \N}\|v^{n_k}\|_{\Cc^{0,1}_{\alpha/2, \alpha}([t,\infty)\times \R^d)}\leq A^* \Psi(t).
	$$
	Then it follows from Part (i) that,  
	\be\label{eq:lm.vinftytvanish}   
	\lim_{t\to \infty} \|v^{\infty}\|_{\Cc^{0,1}_{\alpha/2, \alpha}([t,\infty)\times \R^d)}\leq \lim_{t\to\infty}\sup_{k\in \N}\|v^{n_k}\|_{\Cc^{0,1}_{\alpha/2, \alpha}([t,\infty)\times \R^d)}=0.
	\ee 
	
	Fix an arbitrary $(t,x)\in \T\times {\R^d}$. By the fact $b^{\pi^\infty}\in L^\infty({\R^d})$ and conditions on $\sigma$ stated in Assumption \ref{assume.b.sigma}, the SDE $$
	dX^{\pi^\infty}_s=b^{\pi^\infty}(X_s^{\pi^\infty})dt+\sigma(X^{\pi^\infty}_s)dW_s\quad \text{with }X^{\pi^\infty}_0=x
	$$ 
	admits a unique strong solution (see \cite[Theorem 1]{veretennikov1981strong}) with a density function belonging to $L^q([0,T]\times \R^d)$ for any $T>0$ with $q=\frac{d+2}{d+1+\alpha}=\frac{p}{1-p}$ (see, e.g., \cite[Theorem 9.1.9]{stroock2007multidimensional}).\footnote{The result in \cite[Theorem 9.1.9]{stroock2007multidimensional} is more general that the density belongs to $L^q(\T\times \R^d)$ for all $q>1$.}
	%{\color{blue}Moreover, $X^{\pi^\infty}_s$ has a locally H\"older continuous density for any $s>0$. (not sure if the density is H\"older continuous)} 
	Given an arbitrary $\eps>0$, thanks to \eqref{eq:assume.integralr} and \eqref{eq:lm.vinftytvanish}, there exits some $T_0>0$ such that 
	\be\label{eq:lm.verify0}  
	\left| E_x\left[ \int_{T-t}^\infty r^{\pi^\infty}(t+s,X^{\pi^\infty}_s) d  \right] \right| + \|v^{\infty}\|_{L^\infty([T,\infty)\times \R^d)}<\frac{\eps}{3}\quad \forall T\geq T_0.
	\ee  
	Define $\rho_N:=\inf_{s\geq 0}\{ X^{\pi^\infty}_s\notin B_N(x)\}$, then there exists $N_0>0$ such that $ 
	\P_x(\rho_N\leq T_0)\leq \frac{\eps}{3A^*}, 
	$
	for all $N\geq N_0$. Then, it follows from \eqref{eq:lm.verify0} that
	\be\label{eq:lm.verify1}   
	\begin{aligned}
		&\E_x[v^{\infty}(N\wedge \rho_N, X^{\pi^\infty}_{N\wedge \rho_N-t})]\\
		&= \E_x[v^{\infty}(N\wedge \rho_N, X^{\pi^\infty}_{N\wedge \rho_N-t}){\mathbf 1}_{\{\rho_N \leq T_0 \}}] +\E_x[v^{\infty}(N\wedge \rho_N, X^{\pi^\infty}_{N\wedge \rho_N-t}){\mathbf 1}_{\{\rho_N \geq T_0 \}}]\\
		&\leq  \|v^\infty\|_{L^\infty(\T\times \R^d)} \P_x(\rho_N \leq T_0)+ \|v^{\infty}\|_{L^\infty([T,\infty)\times \R^d)}\leq A^*\frac{\eps}{3A^*}+\frac{\eps}{3}=\frac{2\eps}{3} \quad \forall N\geq N_0\vee T_0.
	\end{aligned}
	\ee 
	%$$v^{\pi^\infty}(t,x)-\E_x\left[ \int_0^{T\wedge \rho_{N}-t} r^{\pi^\infty}(t+s,X^{\pi^\infty}_s) d  \right]\leq \eps,$$
	Take any $N\in\N$ with $N\geq T\vee N_0$. In view of $v^\infty\in W^{1,2}_p(D_N)\cap \Cc^{0,1}_{\alpha/2,\alpha}(D_N)$, 
	we can apply It\^o-Krylov formula (see e.g. \cite[Section 2.10, Theorem 1]{krylov1980controlled})  to $v^\infty(t+s, X^{\pi^\infty}_s)$ to obtain
	\begin{align*}
		&dv^\infty(t+s, X^{\pi^\infty}_s)\\
		&= \left(\partial_t v^\infty(t+s, X^{\pi^\infty}_s)+ b^{\pi^\infty}(x) D_x v^\infty(t+s, X^{\pi^\infty}_s)+\frac{1}{2}\tr\bigg((\sigma\sigma^T)(X^{\pi^\infty}_s) D^2_x v^\infty(t+s, X^{\pi^\infty}_s) \bigg)\right) ds\\
		&\qquad+ \sigma(X^{\pi^\infty}_s)D_x v^\infty(t+s, X^{\pi^\infty}_s) dW_s.
	\end{align*}
	Denote $p^{x}(s,y):= \P_x(X^{\pi^\infty}_s\in dy | s\leq \rho_N)$ and note that $p^x\in L^q(D_N)$.\footnote{Since $C:=\P^x(N\leq \rho_N)$ is a finite value for a fixed $N$, and $\P^x(s\leq \rho_N)\geq \P^x(N\leq \rho_N)$ for all $s\in[0,s]$, we have $\int_{D_N} |p^x(d,y)|^qdsdy=\int_{D_N}\left| \frac{\P^x(X^{\pi^\infty}_s\in dy, s\leq \rho_N)}{\P^x(s\leq \rho^N)}\right|^qds dy\leq \frac{1}{C^q}\int_{D_N} |\P^x(X^{\pi^\infty}_s\in dy)|^qdsdy<\infty$.} %the density of $X^{\pi^\infty}$ before $\rho_N$ of $X^{\pi^\infty}$ with $X^{\pi^\infty}_0=x$. 
	%that the $Y_s:=\int_0^s \sigma(X^{\pi^\infty})D_x v^\infty(t+r, X^{\pi^\infty}_r)dW_r$ is a martingale, 
	Integrating both sides of  above equality over $[0, N\wedge \rho_N-t]$ and then taking expectations  leads to
	$$
	\begin{aligned}
		&v^{\infty}(t,x)- \E_x[v^{\infty}(N\wedge \rho_N, X^{\pi^\infty}_{N\wedge \rho_N-t})]\\	
		&=\E_x\bigg[  \int_0^{N\wedge \rho_N-t} -\bigg(v^\infty_t(t+s, X^{\pi^\infty}_s)+ b^{\pi^\infty}(X^{\pi^\infty}_s) D_x v^\infty(t+s, X^{\pi^\infty}_s)\\
		&\qquad \qquad \qquad \qquad \qquad +\frac{1}{2}\tr\left((\sigma\sigma^T)(X^{\pi^\infty}_s) D^2_x v^\infty(t+s, X^{\pi^\infty}_s) \right) \bigg) ds \\
		%	&+\E_x\left[ \int_0^T \sqrt{(\sigma^2)^{\pi^*}(X^{\pi^*}_s)}v^\infty_{x}(t+s, X^{\pi^*}_s)dW_s \right]\\
		%&\qquad +v^\infty(t+T, X^{\pi^\infty}_T) \bigg]\\
		&= \int_{D_N} -\bigg[v^\infty_t(s, y)+ b^{\pi^\infty}(y) D_x v^\infty(s, y)+\frac{1}{2}\tr\left((\sigma\sigma^T)(y) D^2_x v^\infty(s, y) \right)\bigg] \bigg(\mathbf{1}_{\{s\geq t\}}p^{x}(s-t,y)\bigg)dy ds\\
		&=\lim_{k\to\infty} \int_{D_N} -\bigg[v^{n_k}_t(s, y)+ b^{\pi^{n_k}}(y) D_x v^{n_k}(s, y)+\frac{1}{2}\tr\left((\sigma\sigma^T)(y) D^2_x v^{n_k}(s, y) \right)\bigg] \bigg(\mathbf{1}_{\{s\geq t\}}p^{x}(s-t,y)\bigg)dy ds\\
		&=\lim_{k\to\infty} \int_{D_N}r^{\pi^{n_k}}(s, y)\bigg(\mathbf{1}_{\{s\geq t\}}p^{x}(s-t,y)\bigg)dy ds\\
		&=\int_{D_N}r^{\pi^{\infty}}(s, y)\bigg(\mathbf{1}_{\{s\geq t\}}p^{x}(s-t,y)\bigg)dy ds\\
		&= \E_x\left[  \int_0^{N\wedge \rho_N-t} r^{\pi^\infty}(t+s,X^{\pi^\infty}_s) ds \right],
	\end{aligned}
	$$
	where the third equality follows from Part (i) and the convergence of $b^{\pi^{n_k}}$ to $b^{\pi^\infty}$ in \eqref{eq:lm.brconverge}, the forth equality follows from \eqref{eq:HJBn'} and $\|\lambda_{n_k}\Hc(\pi^{n_k})\|_{L^\infty(\R^d)}\to 0$, and the fifth inequality follows from the convergence of $r^{\pi^{n_k}}$ to $r^{\pi^\infty}$ in \eqref{eq:lm.brconverge}. By virtue of \eqref{eq:lm.verify0} and \eqref{eq:lm.verify1}, the result above implies that 
	$$
	\begin{aligned}
		&|v^{\infty}(t,x)- V^{\pi^\infty}(t,x)|\\
		&=\left|   \E_x[v^{\infty}(N\wedge \rho_N, X^{\pi^\infty}_{N\wedge \rho_N-t})]+\E_x\left[  \int_0^{N\wedge \rho_N-t} r^{\pi^\infty}(t+s,X^{\pi^\infty}_s) ds \right]-V^{\pi^\infty}(t,x)\right|\\
		&\leq \left|   \E_x[v^{\infty}(N\wedge \rho_N, X^{\pi^\infty}_{N\wedge \rho_N-t})] \right| +\left| E_x\left[ \int_{N\wedge \rho_N-t}^\infty r^{\pi^\infty}(t+s,X^{\pi^\infty}_s) d  \right] \right|\leq \eps,
	\end{aligned}
	$$
	for an arbitrary $\eps$. Consequently, $v^{\infty}(t,x)= V^{\pi^\infty}(t,x)$ holds as desired. Moreover, following the same arguments for \eqref{eq:prop.Vt}--\eqref{eq:prop.tnorm}, we conclude that 
	$$
	\sup_{x\in \R^d}\|v^{\infty}(\cdot, x)\|_{\Cc^1_{\alpha/2}(\T)}\leq A^*,
	$$
	which completes the proof. 
\end{proof}

\begin{proof}[\bf Proof of Theorem \ref{thm:equi.existence}:] We now show that $\pi^\infty$ in Lemma \ref{lm:thm.C12andyoung} satisfies the statement in Theorem \ref{thm:equi.existence}. It is directly from Lemma \ref{lm:thm.C12andyoung} that $v^\infty=V^{\pi^\infty}\in \Cc^{0,1 }_{\alpha/2, \alpha}(\T\times \R^d)\cap  W^{1,2, \ul}_{p}(\T\times \R^d)$ with $\sup_{x\in\R^d}\|v^\infty(\cdot,x)\|_{\Cc^{1}_{\alpha/2}(\T\times \R^d)}<\infty$. Hence, it suffices to show that $v^\infty$ is a strong solution to the system: %\eqref{eq:cha.weakHJB}--\eqref{eq:cha.weakHJB'} by replacing $\pi^*$ with $\pi^\infty$.
	\begin{align}
		0=&\partial_t u(0,x)+\frac{1}{2}\tr\left(\sigma\sigma^T(x) D^2_x u(0,x)\right) \notag\\ 
		&+\sup_{\varpi\in \Pc(U)} \left\{ \int_U \left[ b(x,a)D_x u(0,x) +r(0,x,a) \right] \varpi(da) \right\},  \; \text{a.e on $\R^d$},\label{eq:cha.weakHJBlimit}\\
		0=&\partial_t u (t,x)+\frac{1}{2}\tr\left((\sigma\sigma^T)(x) D^2_x u (t,x)\right) +b^{\pi^\infty}(x)D_x u(t,x) +r^{\pi^\infty}(t,x) \label{eq:cha.weakHJBlimit'} \; \text{a.e. on $\T\times\R^d$}.
	\end{align}
	
	We first show that $v^\infty$ is a strong solution to \eqref{eq:cha.weakHJBlimit'} by taking limit in \eqref{eq:HJBn'}. Notice that $v^\infty\in \Cc^{0,1 }_{\alpha/2, \alpha}(\T\times \R^d)\cap  W^{1,2, \ul}_{p}(\T\times \R^d)$ and $\sup_{x\in \R^d}\|v^{\infty}(\cdot, x)\|_{\Cc^1_{\alpha/2}(\T)}\leq A^*$.
	Fix an arbitrary $(t_0, x_0)\in \T\times \R^d$. A direct calculation gives that
	\bee  
	\begin{aligned}	
		0 =&\lim_{k\to\infty}\int_{D_1(t_0, x_0)} \bigg( \partial_t v^{n_k} (t,x)+\frac{1}{2}\tr\left((\sigma\sigma^T)(x) D^2_x v^{n_k} (t,x)\right)\\
		&\qquad \qquad\qquad\qquad +b^{\pi^{n_k}}(x)D_x v^{n_k} (t,x) +r^{\pi^{n_k}}(t,x) +\lambda\delta(t)\Hc(\pi^{n_k}(x)) \bigg) \phi(t,x) dtdx\\
		=&	\int_{D_1(t_0, x_0)} \bigg( \partial_t v^\infty (t,x)+\frac{1}{2}\tr\left((\sigma\sigma^T)(x) D^2_x v^\infty  (t,x)\right)\\
		&\qquad\qquad\qquad+b^{\pi^\infty}(x)D_x v^\infty (t,x) +r^{\pi^\infty}(t,x) \bigg) \phi(t,x) dtdx, \quad \forall \phi\in L^q(D_1(t_0, x_0)),
	\end{aligned}
	\eee 
	where the second equality follows from $\lim_{k\to \infty}\|\lambda_{n_k}\Hc(\pi^{n_k})\|_{L^\infty(\R^d)}=0$, Lemma \ref{lm:thm.C12andyoung} Part (i) and \eqref{eq:lm.brconverge} in Lemma \ref{lm:thm.C12andyoung} Part (ii).
	%time both sides of \eqref{eq:HJBn'} with $\phi$ then taking limit along the index subsequence $(n_k)_{k\in\N}$ yields
	%\bee  
	%\begin{aligned}
	%\int_{D_1(t_0, x_0)} \bigg( &\partial_t v^\infty (t,x)+\frac{1}{2}\tr\left((\sigma\sigma^T)(x) D^2_x v^\infty  (t,x)\right)\\
	%&+b^{\pi^\infty}(x)D_x v^\infty (t,x) +r^{\pi^\infty}(t,x) \bigg) \phi(t,x) dtdx =0\quad \forall \phi\in L^q(D_1(t_0, x_0)).
	%\end{aligned}
	%\eee 
	Then by the arbitrariness of $(t_0, x_0)$, we conclude that $v^\infty$ is a strong solution to \eqref{eq:cha.weakHJBlimit'}.
	
	Next we show $v^\infty$ satisfies \eqref{eq:cha.weakHJBlimit} by contradiction. We start with some a priori estimates involving $(v^{n_k}, \pi^{n_k})_{k\in \N}$ and $(v^\infty, \pi^\infty)$ near $t=0$. For each $n\in \N\cup\{\infty\}$, define
	\begin{align*}
		f^n(x,a):=b(x,a)D_x v^n(0,x)+r(0,x,a),\quad A^n(x):= \max_{a\in U} f^n(x,a).
	\end{align*}
	%For any $x\in \R^d$, $f^n(x,a)$ is equi-continuous due to the uniform bound of $\|v^n\|_{\Cc^{1,2}}$ and uniform Lipschitz of $b,\sigma, r$ on $a$. Then we have that $A^n(x)$ is uniformly Lipschitz on $x$({\color{blue} Can we use it?}). 
	First, note the fact that for any equi-continuous function sequence $g^n:U\to \R$ that converges to a limit $g^\infty(a)$, it holds that
	$$
	\lim_{n\to\infty}\lambda_n \ln\left\{ \int_U \exp\left[ \frac{1}{\lambda_n} g^n(a)\right]da \right\} = \max_{a\in U} g^\infty(a).
	$$
	Therefore, for any $\phi\in L^1(\R^d)$,
	\bee
	\begin{aligned}
		& \int_{\R^d}\left[b^{\pi^\infty}(x) D_x v^{\infty}(0,x) +r^{\pi^\infty}(0, x) \right]\phi(x)dx\\
		&=\lim_{k\to\infty} \int_{\R^d}\left[ b^{\pi^{n_k}}(x) D_x v^{n_k}(0,x) +r^{\pi^{n_k}}(0, x) +\lambda_{n_k} \Hc(\pi^{n_k}(x))\right] \phi(x)dx \\
		&=\lim_{k\to\infty} \int_{\R^d}\lambda_{n_k} \ln\left\{ \int_U \exp\left[ \frac{1}{\lambda_n} f^{n_k}(x,a)\right]da \right\} \phi(x)dx =\int_{\R^d}A^\infty(x)\phi(x)dx,
	\end{aligned}
	\eee 
	where the first equality follows from Lemma \ref{lm:thm.C12andyoung} Parts (i) and (ii) and 
	% $\|v^{n_k}-v^\infty\|_{\Cc^{0,1}_{\alpha/2,\alpha}(\T\times \R^d)}$ in Lemma \ref{lm:thm.C12andyoung} Part (i), the convergence of $\pi^{n_k}$ to $\pi^\infty$ in Lemma \ref{lm:thm.C12andyoung} Part (ii),
	$\lim_{k\to \infty}\|\lambda_{n_k}\Hc(\pi^{n_k})\|_{L^\infty}(\R^d)=0$. As a result, there exists a zero measure set $L_0\subset \R^d$ such that 
	\be\label{eq:thm.supconverge0} 
	\begin{aligned}
		& b^{\pi^\infty}(x) D_x v^{\infty}(0,x) +r^{\pi^\infty}(0, x)\\
		&= \sup_{\varpi\in \Pc(U)} \left\{ \int_U \left[ b(x,a)D_x v^\infty(0,x) +r(0,x,a) \right] \varpi(da) \right\}, \quad \forall x\in  \R^d\setminus L_0.
	\end{aligned}
	\ee 
	Note that $\pi^\infty$ is independent of $t$. Then by
	\be\label{eq:thm.strong.1}  
	\|v^\infty\|_{\Cc^{0,1 }_{\alpha/2, \alpha}(\T\times \R^d)}\vee\sup_{x\in\R^d}\|v^\infty(\cdot,x)\|_{\Cc^{1}_{\alpha/2}(\T\times \R^d)}\leq A^*,
	\ee 
	we conclude that the three terms, $\partial_t v^\infty (t,x)$, $b^{\pi^\infty}(x)D_x v^\infty (t,x) $, $r^{\pi^\infty}(t,x)$, are $\alpha/2$-H\"older continuous on variable $t$ for each $x$. This together with \eqref{eq:cha.weakHJBlimit'} yields that the term $\frac{1}{2}\tr\left((\sigma\sigma^T)(x) D^2_x v^\infty  (t,x)\right)$ shares the same $\alpha/2$-H\"older continuity on variable $t$ for a.e. $x\in \R^d$. Define 
	\be\label{eq:thm.strong.0'} 
	H(t,x):= \partial_t v^\infty (t,x)+\frac{1}{2}\tr\left((\sigma\sigma^T)(x) D^2_x v^\infty  (t,x)\right)
	+b^{\pi^\infty}(x)D_x v^\infty (t,x) +r^{\pi^\infty}(t,x) 
	\ee 
	We conclude from the above analysis that there exists a zero measure set $L\supset L_0$ and a constant $C$ such that for all $x\in \R^d\setminus L$,
	\be\label{eq:thm.strong.5} 
	|H(t,x)-H(0,x)|\leq Ct^{\alpha/2} \quad \forall t\in[0,1].
	\ee
	
	{\bf (a).} Suppose there exists an open set $B_r(x_0)$ and a constant $A>0$ such that, for all $x\in B_r(x_0)$, 
	$$
	\partial_t v^\infty (0,x)+\frac{1}{2}\tr\left(\sigma\sigma^T(x) D^2_x v^\infty (0,x)\right) +\sup_{\varpi\in \Pc(U)} \left\{ \int_U \left[ b(x,a)D_x v^\infty (0,x) +r(0,x,a) \right] \varpi(da) \right\}\geq A.
	$$
	Then, we can conclude from \eqref{eq:thm.supconverge0} and \eqref{eq:thm.strong.5} that there exists $0<\eps_0<1$ with
	$$
	H(t,x)\geq H(0,x)-\frac{A}{2}\geq A-\frac{A}{2}=\frac{A}{2}\quad \forall (t,x)\in [0,\eps_0]\times (B_r(x_0)\setminus L),
	$$
	which contradicts with \eqref{eq:cha.weakHJBlimit'}. 
	
	{\bf (b).} On the other-hand side, if there exists an open set $B_r(x_0)$ and a constant $A>0$ such that, for all $x\in B_r(x_0)$, 
	$$
	\partial_t v^\infty (0,x)+\frac{1}{2}\tr\left(\sigma\sigma^T(x) D^2_x v^\infty (0,x)\right) +\sup_{\varpi\in \Pc(U)} \left\{ \int_U \left[ b(x,a)D_x v^\infty (0,x) +r(0,x,a) \right] \varpi(da) \right\}\leq -A.
	$$
	Then, we can conclude from \eqref{eq:thm.strong.5} that there exists $0<\eps<1$ with
	$$
	H(t,x)\leq H(0,x)+\frac{A}{2}\leq -A+\frac{A}{2}=-\frac{A}{2}\quad \forall (t,x)\in [0,\eps_0]\times (B_r(x_0)\setminus L),
	$$
	which again contradicts with \eqref{eq:cha.weakHJBlimit'}. Then cases (a) and (b) together show that \eqref{eq:cha.weakHJBlimit} holds, which completes the proof.
\end{proof}

\subsection{Proof of Theorem \ref{thm:equi.existence1}}
We finally prove the sufficient condition for the existence of equilibrium in Theorem \ref{thm:equi.existence1}. As a result, the limiting feedback strategy $\pi^\infty$ is an equilibrium that satisfies Definition \ref{def:equi.relaxed}.

\begin{proof}[\bf Proof of Theorem \ref{thm:equi.existence1}:] 
	
	%$$
	%\begin{aligned}
	%0=& \partial_t V^{\pi*}(0,x)+\frac{1}{2}\tr\left(\sigma\sigma^T(x) D^2_x V^{\pi*}(0,x)\right) \notag\\ 
	%	&+\sup_{\varpi\in \Pc(U)} \left\{ \int_U \left[ b(x,a)D_x V^{\pi*}(0,x) +r(0,x,a) \right] \varpi(da) \right\} \text{ a.e. on $\R^d$}.
	%\end{aligned}
	%$$
	Take a feedback strategy $\pi^*:\R\to \Pc(U)$. Suppose $V^{\pi^*}\in \Cc^{0,1 }_{\alpha/2, \alpha}(\T\times \R^d)\cap  W^{1,2, \ul}_{p}(\T\times \R^d)$ with $\sup_{x\in\R^d}\|V^{\pi^*}(\cdot,x)\|_{\Cc^{1}_{\alpha/2}(\T\times \R^d)}<\infty$, and $V^{\pi^*}$ is a strong solution to the EHJB system \eqref{eq:cha.weakHJB}--\eqref{eq:cha.weakHJB'}. Then by the same argument between \eqref{eq:thm.strong.1} and \eqref{eq:thm.strong.5},  we conclude that 
	\be\label{eq:thm.strong.5'} 
	\begin{aligned}
		&\text{The terms $\partial_t V^{\pi^*} (t,x)$, $D_xV^{\pi*} (t,x)$, $\frac{1}{2}\tr\left((\sigma\sigma^T)(x) D^2_xV^{\pi*}  (t,x)\right)$}\\
		&\text{are $\alpha/2$-H\"older continuous on variable $t$ for a.e. $x\in \R^d$},
	\end{aligned}
	\ee
	and  \eqref{eq:thm.strong.5} holds by replacing $v^\infty$ and $\pi^\infty$ in \eqref{eq:thm.strong.0'} with $V^{\pi^*}$ and $\pi^*$ respectively. 
	
	We first show 
	\be\label{eq:thm.strong.7'} 
	\sup_{\varpi\in \Ac}  b^{\varpi}(x) D_x V^{\pi^*}(0,x) +r^{\varpi}(0, x) 
	= b^{\pi^*}(x) D_x V^{\pi*}(0,x) +r^{\pi^*}(0, x) \text{ a.e. on $\R^d$}.
	\ee 
	Suppose an open set $B_r(x_0)$ and $A>0$ exist such that 
	$$
	\sup_{\varpi\in \Ac}  b^{\varpi}(x) D_x V^{\pi^*}(0,x) +r^{\varpi}(0, x) -A
	\geq  b^{\pi^*}(x) D_x V^{\pi*}(0,x) +r^{\pi^*}(0, x) \text{ a.e. on $\R^d$}.
	$$
	Then by \eqref{eq:thm.strong.5'} and $V^{\pi^*}$ satisfying \eqref{eq:cha.weakHJB}, we conclude that there exists an constant $0<\eps_0\leq 1$ such that
	$$
	\begin{aligned}
		&\partial_t V^{\pi^*} (t,x)+\frac{1}{2}\tr\left((\sigma\sigma^T)(x) D^2_x V^{\pi^*}  (t,x)\right) +b^{\pi^*}(x)D_x V^{\pi^*} (t,x) +r^{\pi^*}(t,x)\\ 
		&= H(t)\leq H(0)+\frac{A}{2} \leq -A+\frac{A}{2}=-\frac{A}{2}\quad \text{a.e. on $[0,\eps_0]\times B_r(x_0)$,}
	\end{aligned}
	$$
	which contradicts with the assumption that $V^{\pi^*}$ satisfies \eqref{eq:cha.weakHJB'}. Thus, \eqref{eq:thm.strong.7'} holds.
	
	Now we show that $\pi^*$ satisfies \eqref{eq:def.equi}.
	Define $\Ac:=\{ \varpi:\R^d\to \Pc \text{ is measurable}\}$. There exists a finite constant $C>0$ such that for a.e. $x\in \R^d$,
	\be\label{eq:thm.supconverge1}
	\begin{aligned}
		&\sup_{\varpi\in \Ac} b^{\varpi}(x) D_x V^{\pi^*}(t,x) +r^{\varpi}(t, x)\leq \sup_{\varpi\in \Ac}  b^{\varpi}(x) D_x V^{\pi^*}(0,x) +r^{\varpi}(0, x) +Ct^{\alpha/2}\\
		&=  b^{\pi^*}(x) D_x V^{\pi*}(0,x) +r^{\pi^*}(0, x)+Ct^{\alpha/2}\\
		&\leq  b^{\pi^*}(x) D_x V^{\pi*}(t,x) +r^{\pi^*}(t, x)+2Ct^{\alpha/2}, \quad \forall 0\leq t\leq 1.
	\end{aligned}
	\ee 
	where the first and third lines follow from the H\"older continuity of $D_x V^{\pi^*}$ and $r$ with respect to $t$, the second line follows from \eqref{eq:thm.strong.7'}.
	Fix an arbitrary $x\in{\R^d}$ and take an arbitrary $\varpi\in \Ac$. By the fact $b^{\varpi}\in L^\infty({\R^d})$ and conditions on $\sigma$ stated in Assumption \ref{assume.b.sigma}, the SDE 
	\be\label{eq:thm.SDE} 
	dX^{\varpi}_s=b^{\varpi}(X_s^{\varpi})dt+\sigma(X^{\varpi}_s)dW_s\quad \text{with }X^{\varpi}_0=x
	\ee  
	admits a unique strong solution. Moreover, $X^{\varpi}_s$ has a locally H\"older continuous density for any $s>0$. Fix an arbitrary $0<\eps_0<1$ and let $\rho_N:=\inf_{s\geq 0}\{ X^{\varpi}_s\notin B_N(x)\}$. By Assumption \ref{assume.b.sigma},  for any $\epsilon$,  there exists $N_0>0$ such that 
	$$ 
	\P_x(\rho_N\leq \eps_0)\leq \frac{\epsilon}{2K_2}\quad \forall N\geq N_0.
	$$
	Noting again that  $\|V^{\pi*}\|_{L^\infty(\T\times \R^d)}+\int_0^{\infty} \sup_{x\in \R^d, a\in U} |r(s,x,a)|ds \leq 2K_2$, we get that
	\be\label{eq:thm.equi5}   
	\begin{aligned}
		&\left|  \E_x[ V^{\pi*}( {\eps_0}, X^{\varpi}_{ {\eps_0}})]- \E_x[ V^{\pi*}( {\eps_0}\wedge \rho_N, X^{\varpi}_{  {\eps_0}\wedge \rho_N})] \right| \\
		&+\left| \E_x\left[   \int_0^{\eps_0} r^{\varpi}(s, X^\varpi_s)ds   \right]  -   \E_x\left[   \int_0^{{\eps_0}\wedge \rho_N} r^{\varpi}(s, X^\varpi_s)ds   \right]  \right|\\
		&\leq  \left|\E_x\left[ V^{\pi*}(  {\eps_0}, X^{\varpi}_{  \eps_0}){\mathbf 1}_{\{\rho_N > {\eps_0} \}} \right] \right| + \left|\E_x\left[ \int_0^{\eps_0} r^{\varpi}(s, X^\varpi_s){\mathbf 1}_{\{ \rho_N>{\eps_0} \}}ds  \right] \right|\\
		&\leq  \left(\|V^{\pi*}\|_{L^\infty(\T\times \R^d)} +\int_0^{\eps_0} \sup_{x\in \R^d, a\in U} |r(s,x,a)|ds \right)\P_x(\rho_N >  {\eps_0})\leq \epsilon \quad \forall N\geq N_0.
	\end{aligned}
	\ee 
	Take any $N\in\N$ with $N\geq N_0$. In view that $V^{\pi*}\in W^{1,2}_p(D_N)\cap \Cc^{0,1}_{\alpha/2,\alpha}(D_N)$, applying again the generalized It\^o-Krylov formula to $V^{\pi*}( s, X^{\varpi}_s)$ %(see e.g. \cite[Section 2.10, Theorem 1]{krylov1980controlled} in a more general setting) 
	yields
	\begin{align*}
		&dV^{\pi^*}( s, X^{\varpi}_s)\\
		&= \left(\partial_t V^{\pi*}( s, X^{\varpi}_s)+ b^{\varpi}(x) D_x V^{\pi*}( s, X^{\varpi}_s)+\frac{1}{2}\tr\bigg((\sigma\sigma^T)(X^{\varpi}_s) D^2_x V^{\pi*}( s, X^{\varpi}_s) \bigg)\right) ds\\
		&\qquad+ \sigma(X^{\varpi}_s)D_x V^{\pi*}( s, X^{\varpi}_s) dW_s.
	\end{align*}
	%Denote by $p^{x}(s,y):= \P(X^{\varpi}_s\in dy | s\leq \rho_N)$, which belongs to $L^q(D_N)$.\footnote{Indeed, }
	Taking expectations on both sides of above equality yields
	\be\label{eq:thm.equi3} 
	\begin{aligned}
		&\E_x[V^{\varpi\otimes_{\eps_0}\pi^*}(\eps_0\wedge \rho_N, X^{\varpi}_{\eps_0\wedge \rho_N})]-J^{\pi^*}(x)+ \E_x\left[   \int_0^{{\eps_0}\wedge \rho_N} r^{\varpi}(s, X^\varpi_s)ds   \right]  \\
		&=\E_x\bigg[  \int_0^{ \eps_0\wedge \rho_N} \bigg( \partial_t V^{\pi*} ( s, X^{\varpi}_s)+\frac{1}{2}\tr\left((\sigma\sigma^T)(X^{\varpi}_s) D^2_x V^{\pi*}( s, X^{\varpi}_s)  \right)\\
		&\qquad \qquad \qquad\quad + b^{\varpi}(X^{\varpi}_s) D_x V^{\pi*}( s, X^{\varpi}_s)+r^{\varpi}( s,X^{\varpi}_s) \bigg) ds \bigg] \\
		%&= \int_{[0,\eps_0]\times B_N(0)}\bigg[ \partial_t V^{\pi*} (s, y)+\frac{1}{2}\tr\left((\sigma\sigma^T)(y) D^2_x V^{\pi*}(s, y) \right)+ b^{\varpi}(y) D_x V^{\pi*}(s, y) +r^{\varpi}( s,y)\bigg] p^{x}(s,y)dy ds\\
		&\leq  \int_0^{ \eps_0\wedge \rho_N} \bigg[ \partial_t V^{\pi*} (s, X^{\varpi}_s)+\frac{1}{2}\tr\left((\sigma\sigma^T)(X^{\varpi}_s) D^2_x V^{\pi*}(s, X^{\varpi}_s) \right)\\
		&\qquad \qquad \qquad\quad + b^{\pi^*}(y) D_x V^{\pi*}(s, X^{\varpi}_s) +r^{\pi^*}( s,X^{\varpi}_s)+2Cs^{\alpha/2}\bigg]  ds\\
		&\leq \int_{0}^{\eps_0} 2Cs^{\alpha/2}ds=\frac{2C}{1+\alpha/2}\eps_0^{1+\alpha/2},
	\end{aligned}
	\ee 
	where the first inequality follows from \eqref{eq:thm.supconverge1}, and the second last equality follows from that $V^{\pi^*}$ satisfies \eqref{eq:cha.weakHJB'}. Thanks to \eqref{eq:thm.equi5} and \eqref{eq:thm.equi3}, it holds that
	$$
	\begin{aligned}
		J^{\varpi\otimes_{\eps_0}\pi^*}(x) - J^{\pi^*}(x)=\E_x\left[ V^{\pi*}( {\eps_0}, X^{\varpi}_{ {\eps_0}})+ \int_0^{\eps_0} r^{\varpi}(s, X^\varpi_s)ds   \right]  - J^{\pi^*}(x) \leq \frac{2C}{1+\alpha/2}\eps_0^{1+\alpha/2}+\epsilon.
	\end{aligned}
	$$
	By the arbitrariness of $\epsilon$, we have 
	$$J^{\varpi\otimes_{\eps_0}\pi^*}(x) - J^{\pi^*}(x)\leq \frac{2C}{1+\alpha/2}\eps_0^{1+\alpha/2},$$ and hence $\pi^*$ fulfills Definition \ref{def:equi.relaxed}. %And \eqref{eq:thm.strong.7'} further tells that the supremum is achieved by taking $\varpi = \pi^*$ a.e. on $\R^d$. 
	
	%Moreover, as shown in the proof of Theorem \ref{thm:equi.existence} Part (i), $\pi^\infty$ achieved in Lemma \ref{lm:thm.C12andyoung} is an equilibrium.
\end{proof}

\ \\
\textbf{Acknowledgements}:
%Z. Wang  is partially supported by Shandong Excellent Young Scientists Fund Program (Overseas) (Grant No. 2025HWYQ--022) and NSFC (Grant No.12501659).
X. Yu is partially supported by Hong Kong RGC General Research Fund (GRF) under grant no. 15211524, the Hong Kong Polytechnic University research grant under no. P0045654 and the Research Centre for Quantitative Finance at the Hong Kong Polytechnic University under grant no. P0042708.

\ \\
{
\bibliographystyle{plain} 
\bibliography{reference}
}

\end{document}